\documentclass{amsproc}
\usepackage{amssymb, amsmath, amsthm}
\usepackage[all]{xy}
\usepackage[usenames,dvipsnames]{color}
\usepackage{graphics}

\definecolor{laura}{rgb}{.4, 0, .6}
\definecolor{dorette}{rgb}{0,.7,.7}
\definecolor{Carmen}{rgb}{.0, .2, .9}

\newtheorem{thm}{Theorem}[section]

\newtheorem{prop}[thm]{Proposition}

\theoremstyle{definition}
\newtheorem{dfn}[thm]{Definition}
\newtheorem{defn}[thm]{Definition}
\newtheorem{eg}[thm]{Example}

\newtheorem{rmk}[thm]{Remark}

\newcommand{\calC}{\mathcal{C}}

\newcommand{\calG}{\mathcal{G}}
\newcommand{\calH}{\mathcal{H}}
\newcommand{\calI}{\mathcal{I}}
\newcommand{\calK}{\mathcal{K}}
\newcommand{\calL}{\mathcal{L}}
\newcommand{\calM}{\mathcal{M}}

\newcommand{\calSI}{\mathcal{SI}}
\newcommand{\calTB}{\mathcal{TB}}
\newcommand{\calT}{\mathcal{T}}

\newcommand{\U}{\tilde{U}}

\newcommand{\GMap}{{\mbox{\rm GMap}}}

\newcommand{\Map}{\mbox{\rm Map}}

\newcommand{\Orbigrpds}{\mbox{\rm Orbigrpds}}

\newcommand{\G}{\mathcal{G}}

\begin{document}
\date{\today}

\title[Orbispaces and their Mapping Spaces via Groupoids]{Orbispaces and their Mapping Spaces via Groupoids:  \\ A Categorical Approach}
\author[V. Coufal, D. Pronk, C. Rovi, L. Scull, C. Thatcher]{Vesta Coufal, Dorette Pronk, Carmen Rovi, Laura Scull, Courtney Thatcher }

\maketitle

\section{Introduction}

Orbifolds were first introduced by  Satake  (who  called  them V-manifolds),
and later studied  by Conway \cite{Con}, Thurston \cite{T} and others. They were developed as a generalization of manifolds,
and the original approach to their study was based on charts and atlases.
The difference with the orbifold context is that we allow certain singularities:   the local neighbourhoods
are homeomorphic to $U = \U/G$ where $G$ is a finite group acting on an open set
$\U \subseteq {\mathbb R}^n$.
An orbifold $M$ can then be defined via an
orbifold atlas, which is a locally compatible family of charts $(\U, G)$ such that the sets
$\U/G$ give a cover of $M$.    The usual notion of equivalence of atlases through
common refinement is used; details can be found in \cite{sa56,sa57}.
Note that the original definition required that all group actions be effective,
but it is often useful to drop this
requirement; we will not  require that $G$ acts effectively on $\U$.

Working with orbifold atlases is cumbersome, particularly when dealing with maps
between orbifolds.  Therefore an alternate way of representing orbifolds using groupoids
has been developed.  It was shown in \cite{MP} that every smooth orbifold can be represented
by a Lie groupoid.  This representation is not unique, but is determined up to Morita equivalence.
This way of representing orbifolds allows for a natural generalization to orbispaces, without a smooth structure,
via  topological groupoids.  It also allows for a groupoid-based definition of orbifold maps
(originally called the `good' or `generalized' orbifold maps) which works well for homotopy theory, as noted in  \cite{Adem}.

In this paper, we give an accessible introduction to the theory of orbispaces via groupoids.   
We define a certain class of topological groupoids, which we call orbigroupoids.
Each orbigroupoid represents an orbispace, but just as with orbifolds and Lie groupoids,
this representation is not unique:  orbispaces are  Morita equivalence classes of orbigroupoids.
The orbigroupoid category can be used as a basis for developing results about orbispaces.
We will discuss the connection between orbigroupoids and orbispaces, focusing particularly on creating mapping objects for orbispaces which themselves have orbispace structure.    Throughout this paper, we illustrate our definitions and results with numerous examples which
we hope will be useful in seeing how the categorical point of view is used to study these spaces.

The maps between orbispaces  can be defined either in terms of a bicategory of fractions, or in terms of Hilsum-Skandalis maps.    This paper takes the bicategory of fractions approach, providing a more   concrete description of the
mapping space construction.   Related work has been done by Chen \cite{Chen}, using a more atlas-based approach to
representing orbispaces, by Haefliger \cite{Hae} for \'etale groupoids and Hilsum-Skandalis maps
 and by Noohi \cite{Noo} for topological stacks.   In this paper, we lay the groundwork for the bicategory approach and illustrate how it can be used to define a mapping orbispace.  Further results about the properties of this mapping orbispace are given in \cite{PS-tocome}.  

 This paper begins with background sections.  Section \ref{S:topgps} gives the definition of orbigroupoids, and illustrates how these represent orbispaces.  Section \ref{S:hom} defines homomorphisms between orbigroupoids, and also defines natural transformations, creating the bicategory of orbigroupoids.  Section \ref{d:orbispaces}  defines the orbispace category in terms of the bicategory of fractions of orbigroupoids,  again giving examples of how the orbispace category is represented in this fashion.

 Section \ref{S:GMap} shows how to use the definitions of Section \ref{S:hom} to create a topological groupoid representing the maps between orbigroupoids.  We work through several non-trivial examples of the resulting mapping space, showing how the orbispace structure appears in this approach.    We finish this paper by showing in Section \ref{inertia} that  we can recover the inertia groupoid   of \cite{kawasaki} as a mapping orbispace from a specific orbispace, a one-point space with isotropy.

\section{Topological Groupoids and Orbigroupoids}\label{S:topgps}

We begin with the standard definition of a topological groupoid:  a groupoid in the category of topological spaces, where we have spaces instead of sets and all maps are continuous.

\begin{defn}
A {\em topological groupoid} $\calG$ consists of a space of objects $\calG_0$ and a space of arrows $\calG_1$.
The category structure is defined by the following continuous maps.
\begin{itemize}
\item  The {\em source map}, $s:  \calG_1 \to \calG_0$, which gives the domain of each arrow.
\item  The {\em target map}, $t:  \calG_1 \to \calG_0$, which gives the codomain of each arrow.
\item  The {\em unit map}, $u:  \calG_0 \to \calG_1$, which gives the identity arrow on an object.
\item  The {\em composition map},  $m:  \calG_1 \times_{\calG_0} \calG_1 \to  \calG_1$ where $m(g_1, g_2) = g_2 \circ g_1 = g_2 g_1$.
The pullback over $\calG_0$  ensures that  $t(g_1) = s(g_2)$, so that we are only composing the arrows that match up at their ends.
\item  The {\em inverse map}, $i:  \calG_1 \to \calG_1$ defined by  $i(g) = g^{-1}$. \end{itemize}
These maps need to satisfy the usual category axioms, as well as the expected relationship between an arrow and its inverse.
Specifically, we must have:
\begin{itemize}
\item  (identity)  $m(g,u t(g))=g$ and $m(u  s(g),g)=g$,
\item  (associativity)  $m(g_1,m(g_2,g_3))=m(m(g_1,g_2),g_3)$,
\item  (inverses) $m(g,i(g))=u s(g)$ and  $m(i(g),g)=u t(g)$.
\end{itemize}
\end{defn}
We will often use group notation when discussing arrows in $\calG_1$, writing the identity $u(x)$ as $id_x$, multiplication $m(g_1, g_2) $ as $g_2g_1$ and the inverse map $i(g)$ as $g^{-1}$.

Associated to any topological groupoid we have a topological space defined as a quotient of the object space.  
We think of the arrows in the groupoid as identifications, and form the quotient $\calG_0/\sim$ where $x \sim y$ 
if there is an arrow $g:  x \to y$ in $\calG_1$.    We will denote this quotient space by  $\calG_0/\calG_1$. 
In order to keep our topologies reasonable and the  singularities in this quotient space modeled by quotients 
of finite groups, we need to put some restrictions on our topological groupoids.

\begin{defn}\
An {\em orbigroupoid} is a topological groupoid $\calG$ such that the object space $\calG_0$ and the arrow space $\calG_1$ are compactly generated locally compact, paracompact Hausdorff spaces, and for which the source and target maps  are \'etale (i.e. local homeomorphisms), and  the map  $(s,t)\colon \calG_1\rightarrow \calG_0\times \calG_0$ is proper (i.e. the preimage of a compact set is compact).
\end{defn}

An orbigroupoid defines an orbispace, a topological space with orbifold-type singularities but without the smooth structure.
We think of the orbispace as the quotient space of the objects, but with extra structure at the singularities.  Philosophically, we want to identify points,
but also remember how many times they were identified, and in what way.   
The orbigroupoid allows us to do this.  For any point $x$ in the object space $\calG_0$, 
we define the {\em isotropy group} of $x$  to be $G_x = \{ g \in \calG_1 | s(g) = t(g) = x \}$.  
Because we are working in groupoids, if there is an arrow identifying $x$ to $y$, their isotropy groups will be isomorphic.  
So the isotropy is well-defined on points of the quotient space, and the structure of the singularities can be encoded via isotropy information.

\begin{rmk}\label{D:Etaleproper}
The joint conditions of being \'etale and proper immediately  ensure that the isotropy groups have to be finite: the \'etale condition makes them discrete and the 
proper condition then requires them to be finite.   In fact, these conditions give us even more.  
Let $x\in\calG_0$ be any point in the space of objects.
Since the groupoid is \'etale we can find  open neighbourhoods $U_g$ of 
the elements $g\in G_x$ on which both the source and the target maps restrict to homeomorphisms.
Then since the group $G_x$ is finite,  we can shrink these neighbourhoods so that $s(U_g)=t(U_g)$ for each $g$,  and moreover their images are all the same,
i.e., $s(U_g)=s(U_{g'})$ for all $g,g'\in G_x$.   Call the common image $V_x$.
It is shown in \cite{MP} that properness allows us then to further shrink this neighbourhood $V_x$  (if necessary) so that all arrows in $\calG_1$ which have both
their source and target in $V_x$ are in $\bigcup_{g\in G_x}U_g$, i.e., $(s,t)^{-1}(V_x)=\bigcup_{g\in G_x}U_g$ 
and for each $g\in G_x$, both $s$ and $t$ restrict to a homeomorphism from $U_g$ to $V_x$.  Thus, the quotient space of $V_x$ is a quotient by the group action $G_x$ acting on $V_x$, and we see that with these orbigroupoids, we are indeed modelling spaces which are locally the quotients of finite group actions. 
(The proof of this  appears in
the proof of the implication $4\Rightarrow 1$ of Theorem 4.1 of \cite{MP}.   That paper is about orbifolds rather than orbispaces, and 
the groupoid was required to be effective, but this result does not depend on those conditions.)
\end{rmk}

Note that orbispaces are represented by orbigroupoids, but this representation is not unique.
Orbispaces will be defined as Morita equivalence classes of orbigroupoids.  Before we consider this equivalence,
we give some basic examples of orbigroupoids and their quotient spaces, illustrating how some standard orbispaces are represented.

The first example is a manifold without singularities.

\begin{eg}  Consider the sphere $S^2$.  We will build a topological groupoid $\calG$ representing $S^2$ as follows.
Cover $S^2$ with two open disks, $D_1$ and $D_2$, which intersect each other in an annulus.
Let the object space $\calG_0$ be the disjoint union of the two disks.  The arrow space $\calG_1$ needs to encode
the identifications along the annular overlap.  Hence, we get an annulus of arrows $A_1$ identifying points along the
edge of $D_1$ with their corresponding points along the edge of $D_2$.  So the source map takes $A_1$ to the edge of
$D_1$, and the target map takes $A_1$ to the edge of $D_2$.    Similarly, we have the inverse maps with source and target
reversed forming another annulus $A_2$.   To complete the structure of our topological groupoid, we also include
identity arrows, which sit in two disks homeomorphic to $D_1$ and $D_2$.    Thus we have the orbigroupoid of  Figure \ref{fig:sphere} below.   All the source and target maps are just inclusions, so this is clearly  an \'etale and proper groupoid, hence an  orbigroupoid.   The  quotient space of this groupoid is just our original space $S^2$, with no additional non-trivial isotropy information.

\begin{figure}[h]
\includegraphics{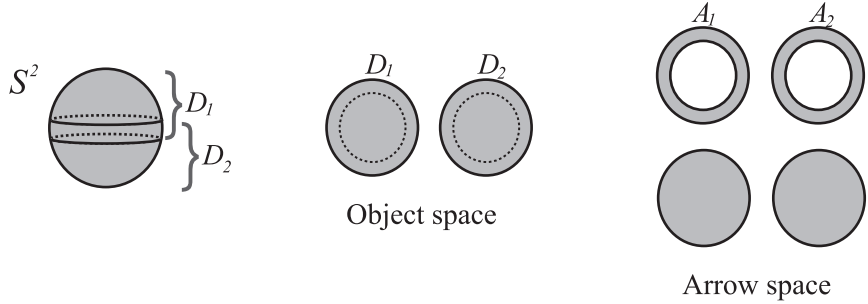}
\caption{2-sphere groupoid}
\label{fig:sphere}
\end{figure}


\end{eg}

This example can easily be generalized to create an orbigroupoid that represents any manifold, which is an example of an orbispace with trivial isotropy.
Next, we look at how to get non-trivial isotropy points via our groupoid representation.
The following example defines a single point with isotropy.

\begin{eg}\label{e:pt}
 Let $G$ be a finite group.  Define the orbigroupoid $*_G$ to have  object space consisting of just one point,
$*$, and arrow space the discrete space $G$.  Composition and inverses are given by the group structure on $G$.  Source and target maps take any arrow $g$ to the point $*$.   Since the object and arrow spaces are finite and discrete, this topological
groupoid is clearly \'etale and proper.    The orbispace associated
to this orbigroupoid is the quotient of the single point $*$ by the group action,  a single point with $G$-isotropy.
\end{eg}

  Note that this is an example of a non-effective orbispace.    We can create a similar space with effective action by a `fat point' construction as follows.

\begin{eg} \label{cone}
Consider the open disk $D^2$ with a $\mathbb{Z}/3$ rotation action that keeps the center point fixed.  
We create a groupoid with the object space equal to the disk itself.  The arrow space encodes the  group action.    
If $\nu$ generates $\mathbb{Z}/3$, then we identify points in an orbit via arrows: $x\to x$, $x\to \nu x$, and 
$x\to \nu^2x$.  Thus the arrow space is three disjoint disks, one for each element of $\mathbb{Z}/3$.  
The source and target maps are defined by the projection and the action respectively: $s(\nu^i,x)=x$ and $t(\nu^i,x)=\nu^ix$, for $i=0,1,2$.
See Figure \ref{fig:cone}.  The inverse arrows are already included,  since $\nu^{-1}=\nu^2$.      
Source and target maps are homeomorphisms from any component, and there are finitely many 
components in the arrow space,  so  this is again \'etale and proper.
The quotient space is topologically a disk with a $\mathbb{Z}/3$-isotropy point at the center.
This orbispace is referred to as an order $3$ cone point in Thurston \cite{T}.  We will denote it  by $\calC_3$.

\begin{figure}[h]
\includegraphics{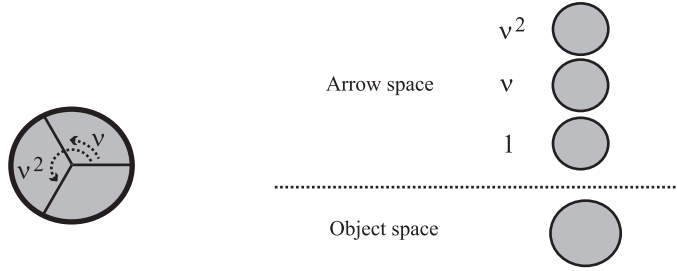}
\caption{Order $3$ cone point groupoid $\calC_3$}
\label{fig:cone}
\end{figure}

\end{eg}

The previous example is a global quotient, defined by the action of a finite group on a space.  Such a global quotient can always be represented by the translation groupoid $G \ltimes X$ where the object space is given by $X$ and the arrow space by $G \times X$, with $s(g,x) = x$ and $t(g, x) = gx$.
The following gives another example of a global quotient orbispace and its translation groupoid.
\begin{eg}\label{e:silvered}
The `silvered interval' is a (closed) interval with $\mathbb{Z}/2$-isotropy at its endpoints.  One way to represent the silvered interval is as the global quotient of a $\mathbb{Z}/2$ action on a circle, with the group acting by reflection and fixing two antipodal points $a$ and $b$.

We describe the translation groupoid $\mathbb{Z}/2 \ltimes S^1$ associated to this action. The object space of $\mathbb{Z}/2 \ltimes S^1$ is $S^1$.
The arrow space is $\mathbb{Z}/2 \times S^1$, the disjoint union of two copies of $S^1$.  The source map is $s(g,x)=x$, and the target map  is $t(g,x)=gx$.  Thus, an arrow of the form $(1 ,x)$ is the identity map $x\to x$, and an arrow of the form $(\tau ,x)$ is a map from $x$ to its reflection $y$.  The inverse of the arrow $(\tau ,x)$ is given by  $(\tau , \tau x)$.   We have non-trivial isotropy maps  $(\tau, a)$ and $(\tau, b)$ creating the isotropy structure on the endpoints of the quotient space.     See Figure ~\ref{fig:si}.
Once again, it is easy to see that this translation groupoid $\mathbb{Z}/2 \ltimes S^1$ is an orbigroupoid
 whose orbispace is the silvered interval.  We will denote this orbigroupoid by $\calSI$.

\begin{figure}[h]
\includegraphics{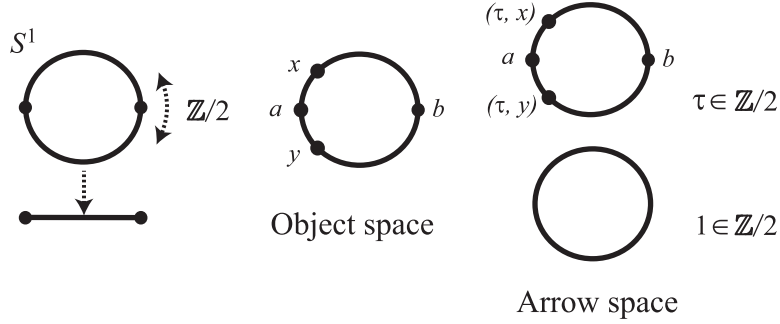}
\caption{Silvered interval}
\label{fig:si}
\end{figure}

 \end{eg}

The next example is not a global quotient.

\begin{eg}\label{teardrop}
The teardrop orbispace is a sphere  $S^2$, with $\mathbb{Z}/3$-isotropy at the north pole.  We represent it with the orbigroupoid $\calT$, created from  an `upper hemisphere'  disk $D_1$ with a $\mathbb{Z}/3$ rotation action (the same as for $\calC_3$ in Example \ref{cone}), and a  `lower hemisphere' disk $D_2$.  The two disks covering the sphere overlap in an annulus around the equator.

The object space $\calT_0$ is the disjoint union of the two upper and lower hemisphere disks, $D_1$ and $D_2$.  The arrow space $\calT_1$ needs to encode both the $\mathbb{Z}/3$ action on the upper hemisphere and the identifications along the annular overlap between the two hemispheres.  As with $\calC_3$, there are three disjoint disks in the arrow space with source and target in $D_1$, corresponding to the identity, $\nu$, and $\nu^2$ actions on points in $D_1$.  There is a single disk corresponding to the identity maps on $D_2$.  For the identification between the disks, the arrow space contains an annulus identifying points in the edge of $D_1$ to their corresponding points in the edge of $D_2$, and a second annulus for the inverse identifications.  See Figure \ref{fig:teardrop}.
These annuli in the arrow space map with the usual inclusion to the upper hemisphere, but via a 3-fold covering to the lower hemisphere (the source map for the first annulus, and the target for the second).  This is not a homeomorphism, but it is a local homeomorphism, giving an orbigroupoid.
\begin{figure}[h]
\includegraphics{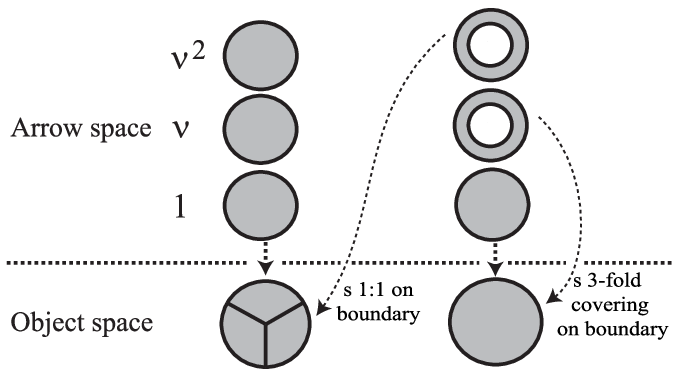}
\caption{Teardrop groupoid $\calT$}
\label{fig:teardrop}
\end{figure}

\end{eg}

The next example gives a somewhat more complicated orbispace.

\begin{eg} \label{billiard}
Let $D$ be an open disk.  The dihedral group $D_3$ (of order 6)  acts on this disk.  We set some notation:  $D_3$ is generated by $\sigma $ and $\rho$, with  $\sigma^2=1$, $\rho^3=1$, $\sigma\rho=\rho^2\sigma$, and $\sigma\rho^2=\rho\sigma$.   Then  $\sigma\in D_3$ acts on the disk by reflection about a chosen line through the center of $D$, and $\rho\in D_3$ acts by counter-clockwise rotation of $D$ by an angle of $2\pi/3$.     
See Figure \ref{fig:cobo}.   The quotient space under this action is  a sector of the disk which we will call a `corner'.  The edges of the sector become `silvered boundaries' with ${\mathbb Z}/2$-isotropy, and the center point becomes a corner point with $D_3$-isotropy.    

\begin{figure}[h]
\includegraphics{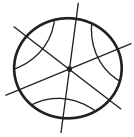}
\caption{Disk with $D_3$ action}
\label{fig:cobo}
\end{figure}

We combine three such corners to create a triangular orbispace where each corner has $D_3$-isotropy, and the edges have $\mathbb{Z}/2$-isotropy; see Figure \ref{fig:three-corners}.   Note that each sector overlaps each of the other sectors, as shown in Figure \ref{fig:three-corners}.   Following naming conventions of Thurston \cite{T}, we will call this the triangular billiard orbispace  $\mathcal{TB}$.

\begin{figure}[h]
\includegraphics{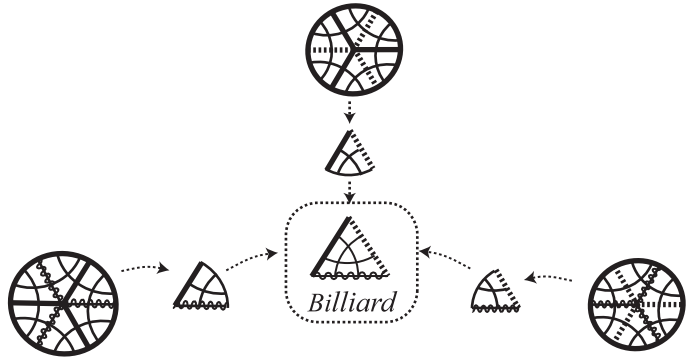}
\caption{Construction of the triangular billiard $\mathcal{TB}$}
\label{fig:three-corners}
\end{figure}

To create a groupoid representing this orbispace, we start with translation groupoids representing each of the three corners as the global quotient $D_3\ltimes D$, and adding in the identifications of the overlaps.   
 So the object space  is the disjoint union of the three disks, one for each sector.  
The arrow space contains six disjoint copies of each of these three disks, corresponding 
to the action of $D_3$ on each disk, with arrows sending $x$ to $gx$ for each $g \in D_3$. 
In addition,  the arrow space contains copies of the overlap shape, representing the glueing 
arrows from one disk to another.    If we consider a space of arrows gluing points in the overlap 
from disk A to those in disk B, we see that there are 6 choices of source embeddings of the overlap in 
disk A (given by three possible images, and a choice of whether to embed with a reflection or not).  There are also 6 possible target embeddings of the overlap in a disk B, again $3$ with reflections and $3$ without.    This gives $36$ possible copies of the overlap, with source in disk A and target in disk B, but  we observe that the map $(i,i)$ (a copy of the overlap with embeddings without reflections)  is the same as  $(\sigma, \sigma)$ (a copy with reflections both into disk A and into disk B), so this reduces the number
to $18$.     See Figure \ref{fig:fish}.
\begin{figure}[h]
\includegraphics{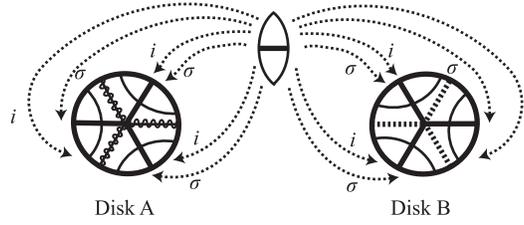}
\caption{Embedding of one of the overlaps}
\label{fig:fish}
\end{figure}
Similarly, there are another $18$ copies of the overlap in the arrow space with source in disk B and 
target in disk A (the inverses of the arrows from disk A to disk B).   In total, we have $36$ overlaps for 
each choice of two of the three disks.  The orbigroupoid $\calTB$ is shown in Figure \ref{fig:TB-groupoid}.
\begin{figure}[h]
\includegraphics{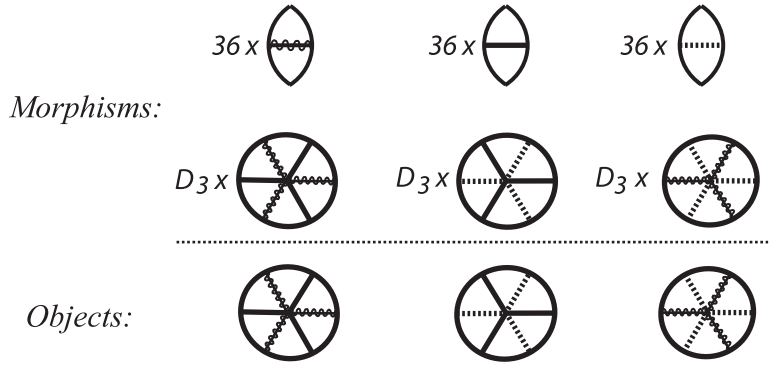}
\caption{Triangular billiard groupoid $\mathcal{TB}$}
\label{fig:TB-groupoid}
\end{figure}
\end{eg}

We emphasize again that  orbigroupoid representations of orbispaces are not unique, and  a single orbispace can be represented by many different orbigroupoids.  Our last example gives an illustration of this.

\begin{eg}\label{e:interval}
Consider the orbispace defined by the interval $I=[0,1]$, with only trivial isotropy.   The simplest representation of this orbispace,  $\calI$, has object space $\calI_0=I$ and arrow space $\calI_1=I$, with the only arrows being identity arrows.    Alternately, we can represent the interval with a `broken' groupoid  $\calI^2$.
The object space $\calI^2_0$ is the disjoint union of two intervals,
$I_L=[0,\frac23)$ and $I_R=(\frac13,1]$.  We glue the subinterval $(\frac13, \frac23)$ (in bold in Figure ~\ref{fig:ints}) on the right end of
$I_L$ to the  subinterval $(\frac13, \frac23)$ (in bold) on the left end of $I_R$ to form $I$.  Hence the arrow
space consists of the identity arrows represented by copies of $I_L$ and $I_R$, and two copies of the  interval
$(\frac13, \frac23)$, one representing the glueing arrows from the subinterval of $I_L$ to the subinterval
of $I_R$, and the other representing their inverses from $I_R$ to $I_L$.
\begin{figure}[h]
\includegraphics{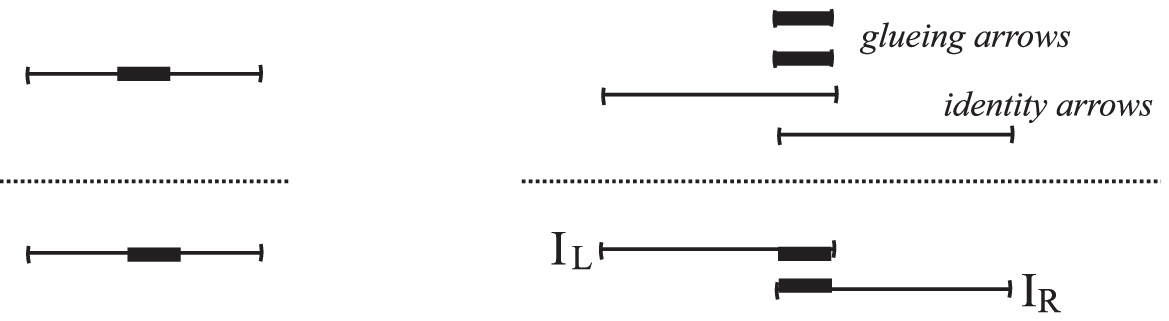}
\caption{Orbigroupoids $\calI$ and $\calI^2$}
\label{fig:ints}
\end{figure}

This can easily be generalized to an orbigroupoid $\calI^n$ representing $I$ in which the interval is broken into $n$ pieces.   For example, we will later use $\calI^3$, where the interval is broken into three pieces.  The object space is $\calI^3_0=[0, \frac12)\coprod (\frac14,\frac34)\coprod (\frac12,1]$.  The arrow space consists of intervals representing the identity arrows, two intervals $(\frac14,\frac12)$ glueing $[0,\frac12)$ to $(\frac14,\frac34)$ and vice versa, and two intervals $(\frac12,\frac34)$ glueing $(\frac14,\frac34)$ to $(\frac12,1]$ and vice versa.
\end{eg}

These examples illustrate how the groupoid keeps track of the ways that points are identified, both to each other and to themselves.  The arrow space encodes both identifications between points in different components of the object space  and local isotropy structure in the same categorical language, with the isotropy structure coming from the ways in which a point is identified to itself.   Orbigroupoids allow us to represent orbispaces by keeping  track not only of what is identified, but the actual number and structure of those identifications.  It is this extra structure that makes orbispaces different from ordinary spaces.

\section{Homomorphisms and Natural Transformations of Orbigroupoids} \label{S:hom}
Next we want to consider maps of orbispaces as represented by homomorphisms of orbigroupoids.

\begin{defn}
      Let $\calG$ and $\calH$ be topological groupoids.
A  {\em homomorphism} $f \colon \calG \to \calH$ is a continuous  functor between topological groupoids.    Specifically, $f$ is defined by two continuous maps,  $f_0\colon \calG_0\rightarrow \calH_0$, and  $f_1\colon \calG_1\rightarrow \calH_1$, which satisfy the functor relations:
\begin{itemize}
\item  If $x\in \calG_0$, then $f_1$ takes the identity map on $x$ in $\calG$ to the identity map on $f_0(x)$ in $\calH$:   $$f_1(u(x))=u(f_0(x)).$$
\item  If $g\in\calG_1$, then $f$ preserves the source and target of $g$:  $$f_0(s(g))=s(f_1(g)),$$  $$f_0(t(g))=t(f_1(g)).$$
\item  If $g_1$ and $g_2$ are two arrows in $\calG_1$ such that $t(g_1) = s(g_2)$, then the arrows $f_1(g_1)$ and $f_1(g_2)$ can be composed in $\calH_1$ since $f$ preserves the source and target.  Moreover,  $f_1$ respects the composition:  $$m(f_1(g_1),f_1(g_2))=f_1(m(g_1,g_2)).$$
\end{itemize}
It follows from the above that $f$ preserves inverses as well.
\end{defn}

A homomorphism of orbigroupoids induces a map between the represented orbispaces:  if two points are identified in $\calG_0$ via some map $g$,
then $f(g)$ identifies their images in $\calH$.  So we get a continuous map on the underlying quotient spaces.
The homomorphism also carries information about the isotropy, since an identification in $\calG$ is mapped
to a specific identification in $\calH$.    The following examples illustrate how the homomorphisms encode  information about how the  map behaves on the isotropy.

\begin{eg}
We examine possible homomorphisms from $*_G$ to the silvered interval $\mathbb{Z}/2 \ltimes S^1$,
our orbispaces from Examples \ref{e:pt} and \ref{e:silvered}.  We will compare the results  for
$G=\mathbb{Z}/3$ and $G=\mathbb{Z}/4$.

For each element $g\in G$, denote the corresponding arrow in $*_G$ by $g$ as well.  Note that if $g$ has order $n$ in the group, then the arrow $g^n$ is the identity arrow in $*_G$.    Since  a
homomorphism $f:*_G\to \mathbb{Z}/2 \ltimes S^1$ respects composition,   
$f_1$ gives a group homomorphism from $G$ to the isotropy group $I_{f(*)}$.  
In particular, if $g^n = 1$, then $(f_1(g))^n $, the composition of $f_1(g)$ with itself $n$ times, must be an identity arrow.

Suppose that $G=\mathbb{Z}/3$, and consider a homomorphism  
$f\colon *_{\mathbb{Z}/3}\to \mathbb{Z}/2 \ltimes S^1$.  Let $y=f_0(*)\in S^1$.   If  $g$ is a non-zero element of $\mathbb{Z}/3$, then
 $g$ has order 3, and $[f(g)]^3=f(g^3)$ is the identity arrow $y \to y$.  The only arrows $h$ with the property that $h^3=1$
have the form  $(1,y)$, and it follows that $f$ must take all arrows to identity arrows.

We get more interesting homomorphisms when $G$ has elements of even order.  Suppose now that $G=\mathbb{Z}/4$ generated by $\sigma$ with $\sigma^4 = 1$,
and that $f:*_{\mathbb{Z}/4}\to \mathbb{Z}/2\ltimes S^1$ is a homomorphism with $y=f_0(*)\in S^1$.
If we want a non-trivial map on arrows, we must have $y$ be
one of our isotropy points $a$ or $b$, and  $f(\sigma)=(\tau ,y)$.  Since $\sigma^3 = \sigma^{-1}$, we also have $f(\sigma^3) = (\tau, y)$.  Thus, we get two
homomorphisms which are non-trivial on
arrows, one with $y=a$ and the other with $y=b$, in addition to the maps that take all arrows to the identity.
\end{eg}

In general, for any finite group $G$ and any orbigroupoid $\calH$, the homomorphisms $f\colon *_G\to\calH$
correspond to a point $f(*)\in \calH_0$ with a group homomorphism into the isotropy group $H_x$ of $x$, $f_1\colon G\to H_x$.

Next we look at how to represent a path in an orbispace, that is, a map from an interval to the orbispace, via a groupoid homomorphism.

\begin{eg}\label{e:ItoB}  Recall $\calI$, an orbigroupoid representing an interval $I$ from Example \ref{e:interval}, and  $\calTB$, the triangular billiard groupoid from Example \ref{billiard}.    We can consider paths in $\calTB$ which can be represented by  homomorphisms $f:\calI\to\calTB$.  Since the object space is connected and $f_0$ is continuous, $f$ will send the entire interval $I$ into one of the disks in the object space of $\calTB$.  Additionally, as the arrow space is also connected, and $f_1$ must be continuous and preserve the identities,  $f_1$ can only send the interval $I=\calI_1$ to the identity component for the disk that $\calI_0$ is mapped into.  So not all paths can be represented by homomorphisms between these particular orbigroupoids.  We will return to this idea in the next section.

\begin{figure}[h]
\includegraphics{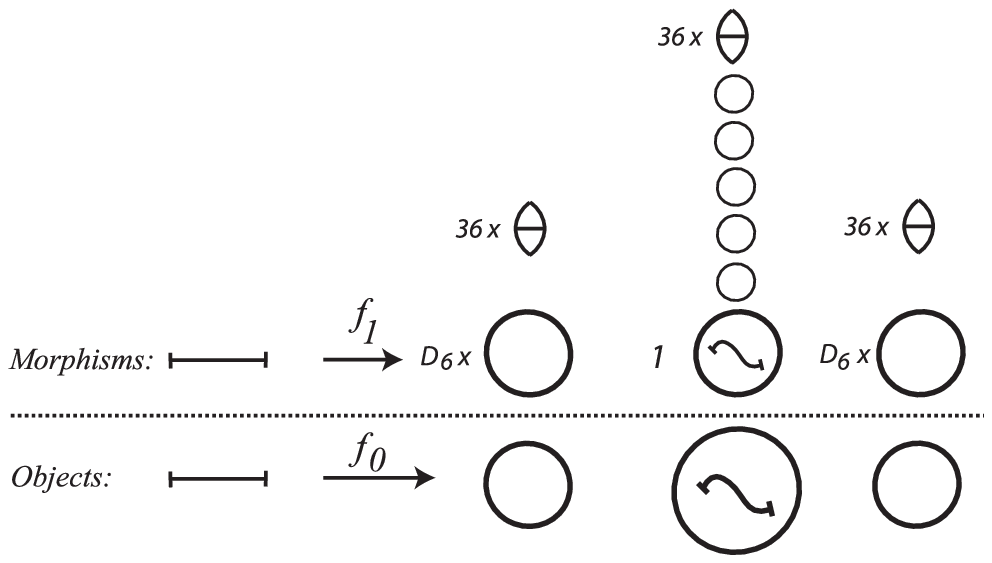}
\caption{Homomorphism $f:\calI \to\mathcal{TB}$  }
\end{figure}
\end{eg}

We have shown that the representation of an orbispace by an orbigroupoid is not unique; the same thing is true for maps of orbispaces.  In order to define the category of orbispaces, we will need to look at identifications of homomorphisms.

\begin{defn}
Given two  homomorphisms $f,f'\colon\mathcal{G}\rightarrow\mathcal{H}$, a 2-cell $\alpha\colon f\Rightarrow f'$ is a continuous natural transformation between these functors.
Specifically, $\alpha$ is given by a continuous function $\alpha\colon G_0\rightarrow H_1$
such that $s\circ\alpha=f_0$, $t\circ\alpha=f'_0$ and, for every arrow $g:  x \to y$ in $\calG_1$, the following naturality square
of arrows in $\calH$ is commutative:
$$
\xymatrix{
f_0(x)\ar[r]^{\alpha(x)}\ar[d]_{f_1(g)} &f'_0(x)\ar[d]^{f'_1(g)}
\\
f_0(y)\ar[r]_{\alpha(y)} & f'_0(y)}
$$
In other words, $m(\alpha(x), f'_1(g))=m(f_1(g), \alpha(y))$.
\end{defn}

Two homomorphisms with a natural transformation between them represent the same map between the quotient spaces, since there is an identification arrow between $f_0(x)$ and $f'_0(x)$ for any $x$.

\begin{eg}  \label{e:I2toSI}
We define two homomorphisms, $f$ and $f'$, from the `broken' interval $\calI^2$ (Example \ref{e:interval}) to the silvered interval $\calSI$ (Example \ref{e:silvered}).   The first one,  $f$ is defined so that $f_0$ takes $I_L$ and $I_R$ into the upper portion of the circle in $\calSI_0$, and $f_1$ takes all arrows to the corresponding identity arrows.  The second, $f'$ is defined so that $f'_0$ is the same as $f_0$ on $I_L$, but takes $I_R$ to the reflection of $f_0(I_R)$ in the lower portion of the circle:   if $f_0(x)=y$ for some $y\in\calSI_0$, then $f'_0(x)=f_0(x)$ for $x\in I_L$, and $f'_0(x)=\tau y=\tau f_0(x)$ for $x\in I_R$.  Then $f'_1$ takes the identity arrows to the identity component, as usual, but takes the glueing arrows in $\calI^2_1$ to the $\tau$ component of $\calSI_1$:  if $g:x\to y$ is a glueing (non-identity) arrow in $\calI^2_0$, then $f_1(g)=(1,f_0(x))$ and $f'_1(g)=(\tau,f'_0(x))$.  See Figure \ref{fig:I2toSI}.  It is easy to check that these maps satisfy the functorial requirements for being homomorphisms.

\begin{figure}[h]
\includegraphics{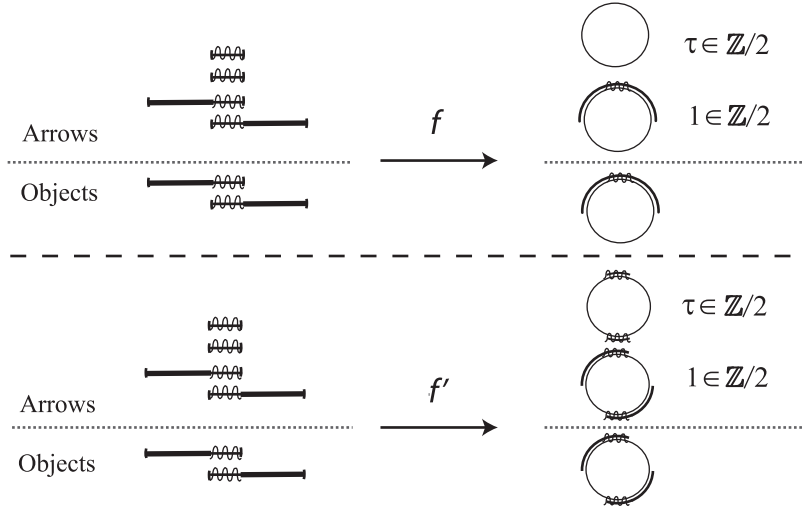}
\caption{Homomorphisms $f$ and $f'$ from $\calI^2$ to $\calSI$ }
\label{fig:I2toSI}
\end{figure}
Now we define a natural transformation $\alpha:f\Rightarrow f'$.  Let $\alpha:\calI^2_0\to\calSI_1$ be such that $\alpha(x)=(1,f_0(x))$ for $x\in I_L$ and $\alpha(x)=(\tau,f_0(x))$ for $x\in I_R$.  Then for $x\in I_L$, $s\circ\alpha(x)=s(1,f_0(x))=f_0(x)$ and $t\circ\alpha(x)=t(1,f_0(x))=f_0(x)=f'_0(x)$.  For $x \in I_R$,
$s\circ\alpha(x)=s(\tau,f_0(x))=f_0(x)$ and $t\circ\alpha(x)=t(\tau,f_0(x))=\tau f_0(x)=f'_0(x)$.  So $\alpha(x)$ is indeed an arrow $f_0(x) \to f'_0(x)$ for each $x$. 

Finally, we need to check that the naturality square commutes for each $g\in\calI^2_1$.  This is clear for any identity arrow of $I_L$, since all arrows are the identity.   So we can consider the commutative squares for an identity arrow $id_x$  with $x\in I_R$, on the left, and a glueing arrow $g:x\to y$ with  $x\in I_L$ and $y\in I_R$, on the right:
\[  \xymatrix {f_0(x) \ar[r]^{(\tau,f_0(x))}\ar[d]_{(1,f_0(x))} &
{\tau f_0(x)}\ar[d]^{(1,\tau f_0(x))} &\qquad &
f_0(x) \ar[r]^{(1,f_0(x))}\ar[d]_{(1,f_0(x))} &
f_0(x)\ar[d]^{(\tau, f_0(x))}\\
 f_0(x) \ar[r]_{(\tau ,f_0(x))} &
{\tau  f_0(x)} & &
 f_0(x) \ar[r]_{(\tau,f_0(x))} &
 \tau f_0(x)
 } \]
 \noindent and see that both these squares commute.    
The final case, for a glueing arrow $g:x\to y$ with $x\in I_R$ and $y\in I_L$ is similar to the right square.
 Notice that  both $f$ and $f'$ map the interval continuously onto the same path in the quotient space.
\end{eg}

Even though homomorphisms with a natural transformation between them can be thought of as `the same' map of 
orbispaces, we do not want to simply identify them.   In examining the structure of maps between orbispaces, 
we will also  want to take into consideration  how many ways homomorphisms are identified.    The ways in 
which a homomorphism can be identified to itself via a natural transformation can be used to give the maps themselves isotropy structure.    In order to retain this singularity structure, 
we will be remembering identifications and working with the 2-category of orbigroupoids,  together with 
their homomorphisms and natural transformations.

\section{Representing Orbispaces with Orbigroupoids}\label{d:orbispaces}

Now we begin to  build the category of orbispaces, which is our primary interest.
As demonstrated above, every orbigroupoid defines an orbispace, but this definition is not unique:
it is possible to have  the same quotient space and the same local isotropy structure with
two different groupoid representations.  This ambiguity is made precise via an equivalence of categories, again suitably topologized.
\begin{defn}
An  {\em essential equivalence} of topological groupoids  is a homomorphism $f\colon \calG \to \calH$ satisfying the two conditions:
 \begin{itemize}
\item[E1]
$f$ is {\em essentially surjective on objects} in the sense that $t\circ\pi_2$ is
an open surjection:
$$
\xymatrix@R=1em{
\calG_0\times_{\calH_0}\calH_1\ar[d]_{\pi_1}\ar[r]^-{\pi_2}&\calH_1\ar[d]^s\ar[r]^-t & \calH_0
\\
\calG_0\ar[r]_{f_0} & \calH_0\rlap{ }
}
$$
i.e.,   for each object $y\in \calH_0$, there is an object $x\in\calG_0$ and an arrow
(actually, an isomorphism since we are working with groupoids) $h\in \calH_1$ from $f_0(x)$ to $y$.
\item[E2]
$f$ is {\em  fully faithful} in the sense that the following diagram is a pullback:
$$
\xymatrix@R=1em{
\calG_1\ar[r]^{f_1}\ar[d]_{(s,t)} & \calH_1\ar[d]^{(s,t)}
\\
\calG_0\times \calG_0\ar[r]^{f_0\times f_0} & \calH_0\times \calH_0
}
$$
i.e.,   for any two objects of $\calH_0$
the identifications between them in $\calH_1$ are isomorphic to those between any pre-images in $\calG$.  In particular,  $f$ is an isomorphism on isotropy groups for all objects.\\
\end{itemize}
\end{defn}

Note that condition E1 ensures that an essential equivalence between groupoids produces a homeomorphism on their quotient spaces, and condition E2 means that the isotropy information about the singularities is also preserved.  Therefore an essential equivalence represents an isomorphism on orbispaces.
 Any two orbigroupoids $\calG, \calH$ that can be connected with a zig-zag of essential equivalences $\calG \leftarrow \calK_1 \to \calK_2 \leftarrow \calK_3 \to \cdots \leftarrow \calK_n \to \calH$ are called
Morita equivalent, and represent the same orbispace.  It was shown in \cite{Pr-comp}  that 
in fact two \'etale groupoids $\calG$  and $\calH$ are Morita equivalent if and only if there is a zig-zag of essential equivalences 
$\calG\leftarrow\calK\rightarrow\calH$;  we only need one zig-zag, and any longer list of zig-zags can be shortened.

The next example shows that our various representations of the interval $I = [0,1]$ from Example \ref{e:interval} are Morita equivalent.

 \begin{eg}\label{e:equivinterval}
Recall that we have defined two different orbigroupoids, the `unbroken' $\calI$ and the `broken' $\calI^2$, representing the closed unit interval $I=[0,1]$ (Example \ref{e:interval}).  Here we show that there is an essential equivalence between them.  Let $f:\calI^2\to \calI$ be the homomorphism such that $f_0$ maps $I_L$ in $\calI^2_0$ to
the left side of the interval $\calI_0$, mapping the bold subinterval of $I_L$ to the bold subinterval of
$I_0$, and mapping $I_R$ to  the right side of $\calI_0$ in a similar way.  Since $\calI_1$ consists only of identity arrows,
every arrow is $\calI^2_1$ is mapped to the corresponding identity arrow.  To see
that $f$ is an essential equivalence, first note that $f$ is surjective, and hence essentially surjective
on objects.  To see that  $f$ is fully faithful, let $y, y' \in \calI_0$.  Suppose
$x$ and $x'$ are preimages of $y$ and $y'$, respectively, under $f_0$.  If $y\neq y'$, then there are no arrows either $y\to y'$ or $x\to x'$.    If $y=y'$, then there is only the identity arrow from $y$ to $y'$.  If $x = x'$, then there is again only the identity arrow;  if $x \neq x'$, then $x$ and $x'$ are in different components $I_L$ and $I_R$ of $\calI^2_0$, and again there is  exactly one arrow $g: x\to x'$ which identifies $x$ to $x'$.  In all cases, the sets of arrows in $\calI^2_1$ and $\calI_1$ are isomorphic via $f_1$, since there is always only one.

It is easy to create a similar essential equivalence from  $\calI^n$ to $\calI$, so all of these orbigroupoids are Morita equivalent.
\end{eg}

We create a category of orbispaces from the category of orbigroupoids by inverting the Morita equivalences between orbigroupoids.
The way to do this is via a  bicategory of fractions $\Orbigrpds(W^{-1})$ where $W$ is the class of
essential equivalences \cite{Pr-comp}.  In this construction,
an arrow from $\calG$ to $\calH$ is given by a span of groupoid homomorphisms
$$
\xymatrix{
\calG&\ar[l]_\upsilon\ar[r]^\varphi\calK&\calH}
$$
where $\upsilon$ is an essential equivalence. Such a span is also called a {\em generalized map}
from $\calG$ to $\calH$.

This definition reflects the fact  that there are certain maps between orbispaces that can only be
carried  by certain representing orbigroupoids.
So even though two orbigroupoids may be Morita equivalent, the homomorphisms out of them
are not the same.  In order to represent a map from the orbispace represented by
$\calG$ to the orbispace represented by $\calH$, we may need to replace the original
representing orbigroupoid $\calG$ with an alternate $\calK$ representing the same orbispace,
but which can be used to define the desired map via a span as above.    The following gives an example of such a map.

\begin{eg}\label{ItoT} We would like a map from the interval orbispace $I$ with trivial isotropy to the teardrop,  whose image is a path
crossing from the lower portion of the teardrop to the upper portion, as in Figure \ref{fig:teardrop-path}.
\begin{figure}[h]
\includegraphics{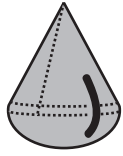}
\caption{Path in the teardrop orbispace}
\label{fig:teardrop-path}
\end{figure}

We represent $I$ by the orbigroupoid $\calI$, and the teardrop
by the orbigroupoid $\calT$ from Example \ref{teardrop}.    However, no homomorphism
$\calI \to \calT$ can produce the desired path in the teardrop,  since the image in  $\calT_0$ is partially in one connected
component and partially in the other and so the map from objects $\calI_0$ would fail to be continuous.
  On the other hand, we can create a homomorphism representing this path  if we represent  the interval orbispace $I$ by the
orbigroupoid $\calI^2$ instead; the desired homomorphism $\calI^2$ to the teardrop is illustrated in
Figure \ref{fig:map-interval-teardrop}, where the overlap in $\calI^2$ allows us to move between components of $\calT_0$.  
Note that the interval in the arrow space $\calI^2$ that represents the glueing of the two components gets sent by the 
homomorphism to arrows in $\calT$ which identify the images of the overlap.  
Thus, this map is represented by a span $\calI \leftarrow \calI^2 \to \calT$, where the map $\calI\leftarrow\calI^2$ 
is the one described in Example \ref{e:equivinterval}.

\begin{figure}[h]
\includegraphics{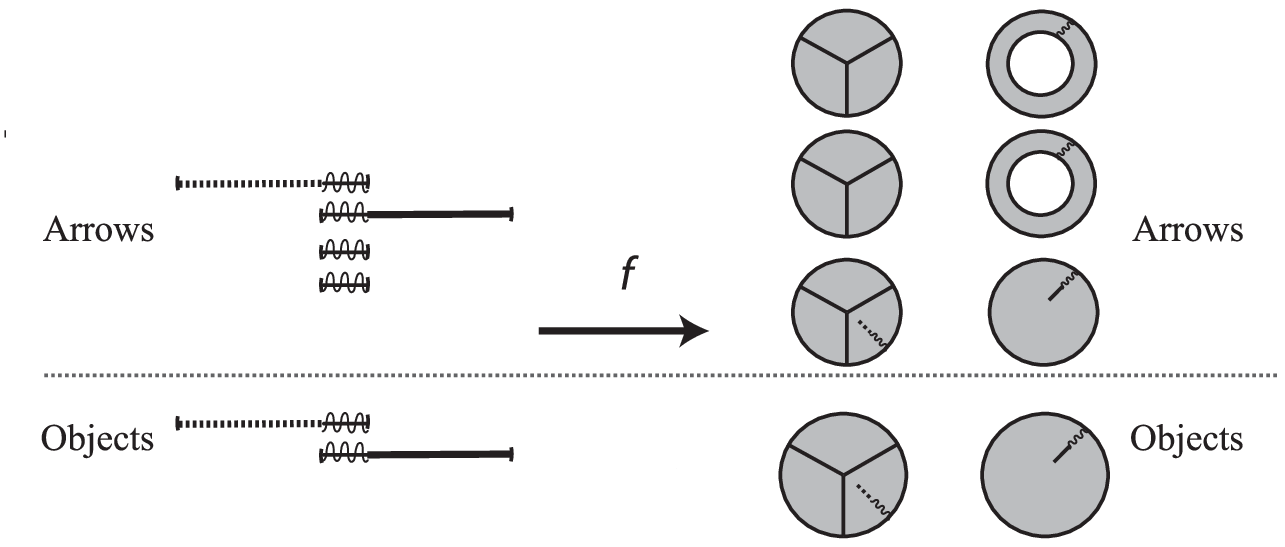}
\caption{Map from $\calI^2$ to $\calT$}
\label{fig:map-interval-teardrop}
\end{figure}

\end{eg}

The choice of representing orbigroupoid matters in the codomain as well, as the next example shows.

\begin{eg} \label{e:codomain}
In Example \ref{e:silvered}, we represented the silvered interval by the translation groupoid $\calSI={\mathbb Z}/2\ltimes S^1$.  We can also represent the silvered interval by a `broken' version similar to the `broken' versions of the interval.   Specifically, $\calSI^2_0$ is the disjoint union of two open intervals, and $\calSI^2_1$ consists of the identity arrows, arrows representing a $\mathbb{Z}/2$-folding action on each interval in the object space, and glueing arrows from the ends of one interval in the object space to the other.  See Figure \ref{fig:SI2}.

\begin{figure}[h]
\centering
\includegraphics{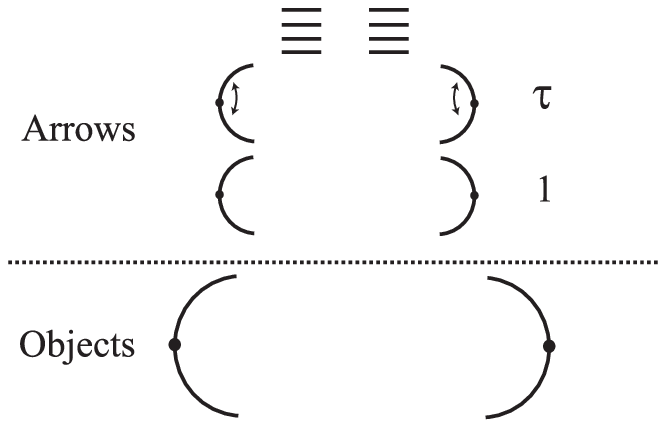}
\caption{Orbigroupoid $\calSI^2$ representing the silvered interval}
\label{fig:SI2}
\end{figure}

Now, consider the path $f:\calI\to\calSI$ in the silvered interval given by mapping the entire interval in $\calI_0$ to the entire upper half of $S^1=\calSI_0$.  Notice that the path in the orbispace includes both endpoints of the orbispace (the two points with $\mathbb{Z}/2$-isotropy).  This path cannot be achieved by a map $f':\calI\to\calSI^2$.  If such a map existed, the continuous map $f'_0:I\to \calSI^2_0$ would have to map the connected interval $I$ to only one connected component of $\calSI^2_0$.  The resulting path in the orbispace would include only one of the endpoints.  Thus, there is no such path $f':\calI\to\calSI^2$.
\end{eg}

Next we need to consider identifications between generalized maps.  A 2-cell between generalized maps is represented by an equivalence class of  diagrams of the following
form
$$
\xymatrix@C=3em{
&\calK\ar[dl]_{\upsilon}\ar[dr]^{\varphi}
\\
\calG\ar@{}[r]|{\alpha\Downarrow} &\calL\ar[u]_{\nu}\ar[d]^{\nu'}\ar@{}[r]|{\beta_\Downarrow} &\calH
\\
&\calK'\ar[ul]^{\upsilon'}\ar[ur]_{\varphi'}
}
$$

The upcoming example gives two generalized maps with a 2-cell between them.

\begin{eg} \  We start with the generalized maps, given by spans.    We saw above in Example \ref{ItoT} that creating a span $\calI \leftarrow \calI^2 \to \calT$ allows us to make a generalized map from $\calI$ with image in two different charts, using the overlap in $\calI^2$ to jump between charts.   Here we show  two spans which put the jump in different places, which  have a 2-cell between them and represent the same map on the quotient spaces.

Recall that in Example \ref{e:equivinterval} we glued together the two intervals in the object space of $\calI^2$ to create an essential equivalence $f:\calI^2\to \calI$.  We do the same thing to $\calI^3$ in two different ways.  We can glue the middle and right intervals in the object space together giving a map that we will call $\nu:\calI^3 \to \mathcal{RI}^2$.  Alternatively, we can glue the left and middle intervals together, producing a map $\nu':\calI^3\to\mathcal{LI}^2$.  This gives two representations of the interval, $\mathcal{RI}^2$ and $\mathcal{LI}^2$, which are broken in different places, with  essential equivalences $\upsilon:\mathcal{RI}^2\to\calI$ and $\upsilon':\mathcal{LI}^2\to\calI$ defined by mapping down to the `unbroken' interval as in Example \ref{e:equivinterval}.

Now we define two maps $\varphi:\mathcal{RI}^2\to\calTB$ and $\varphi':\mathcal{LI}^2\to\calTB$.   On objects, $\varphi$ maps the two `broken' intervals in $\mathcal{RI}^2$ to paths in two different disks, labeled disk $A$ and disk $B$ in Figure ~\ref{fig:phibilliard}, in the object space of $\calTB$.  (Only the pertinent parts of $\calTB$ are shown in the Figure.)    The glueing arrows are mapped to the appropriate overlaps (of the form $i\circ \sigma$ to account for the necessary reflection) in $\calTB_0$ with domain disk $A$ and codomain disk $B$, or the inverse with domain disk $B$ and codomain disk $A$ (see Example \ref{billiard}).  The overlaps will then glue the ends of the two paths together in the resulting orbispace.
\begin{figure}[ht!]
\centering
\includegraphics{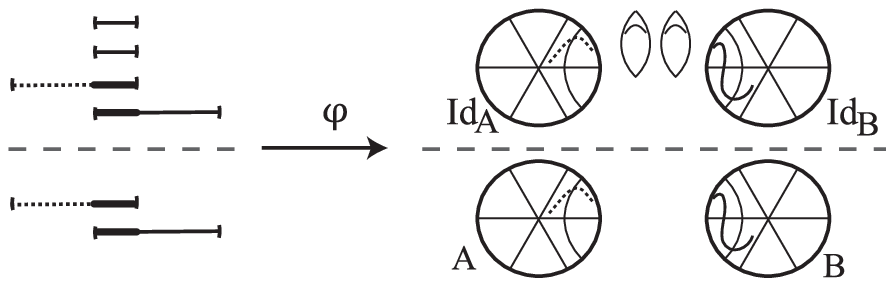}
\caption{$\varphi:\mathcal{RI}^2\to\calTB$}
\label{fig:phibilliard}
\end{figure}

The map $\varphi'$ is similar to $\varphi$ except that the jump between charts occurs in a different place, to match  the different overlap in  $\mathcal{LI}^2$.  See Figure ~\ref{fig:phi'billiard}.
\begin{figure}[ht!]
\centering
\includegraphics{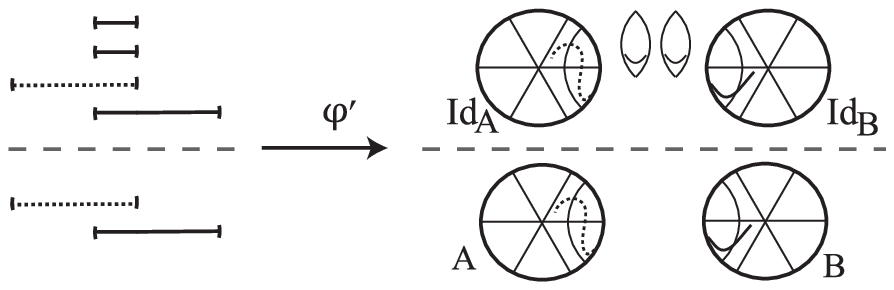}
\caption{$\varphi':\mathcal{LI}^2\to\calTB$}
\label{fig:phi'billiard}
\end{figure}
 This gives us two generalized maps $\calI \to \calTB$ defined by
 $
\xymatrix{
\calI&\ar[l]_\upsilon \ar[r]^\varphi\mathcal{RI}^2 &\calTB}
$
and\\
$
\xymatrix{
\calI&\ar[l]_{\upsilon'} \ar[r]^{\varphi'} \mathcal{LI}^2 &\calTB}
$.
These  produce the same path in the orbispace for $\calTB$, shown in Figure ~\ref{fig:seagull}.
\begin{figure}[ht!]
\centering
\includegraphics{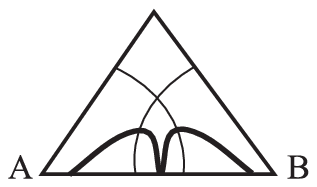}
\caption{Path in $\calTB$}
\label{fig:seagull}
\end{figure}

We create a  2-cell between these two spans to identify them.
$$
\xymatrix@C=3em{
&\mathcal{RI}^2 \ar[dl]_{\upsilon}\ar[dr]^{\varphi}
\\
\calI \ar@{}[r]|{\alpha\Downarrow} &\calI^3 \ar[u]_{\nu}\ar[d]^{\nu'}\ar@{}[r]|{\beta_\Downarrow} &\calTB
\\
&\mathcal{LI}^2 \ar[ul]^{\upsilon'}\ar[ur]_{\varphi'}
}
$$
First, note that $ \upsilon\nu= \upsilon'\nu' : \calI^3_0\to\calI_1$, and so we can fill in the left hand side of the above diagram with the identity natural transformation $\alpha$.
The natural transformation $\beta: \varphi\nu \Rightarrow \varphi'\nu'$ is given by the continuous map $\calI^3_0 \to \calTB_1$ illustrated in Figure ~\ref{fig:phibetaphi'}.  The overlap shown is the one with domain disk $B$ and codomain disk $A$, the same as in the previous two Figures.
\begin{figure}[ht!]
\centering
\includegraphics{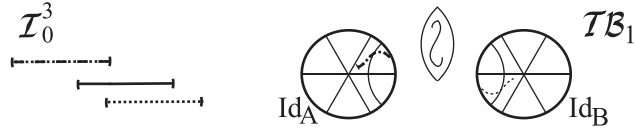}
\caption{Natural transformation $\beta$}
\label{fig:phibetaphi'}
\end{figure}
\end{eg}

We see in this example that a 2-cell between two generalized maps requires a groupoid representation for the domain orbispace
with essential equivalences into the groupoid representations used as the middle of the the two generalized maps.  
This situation is analogous to maps between manifolds, where one can define the maps on conveniently chosen atlases,
but if we want to compare the maps, we need a common refinement.

We consider orbispaces via a 2-category of orbigroupoids with  morphisms given by generalized maps and $2$-cells as described above.    In order to avoid higher structure, we put the following equivalence on the $2$-cells.  
The two diagrams 
$$
\xymatrix@C=3em{
&\calK\ar[dl]_{\upsilon}\ar[dr]^{\varphi} &&&&\calK\ar[dl]_{\upsilon}\ar[dr]^{\varphi}
\\
\calG\ar@{}[r]|{\alpha_1\Downarrow} &\calL_1\ar[u]_{\nu_1}\ar[d]^{\nu'_1}\ar@{}[r]|{\beta_1\Downarrow} &\calH & \mbox{and} &\calG\ar@{}[r]|{\alpha_2\Downarrow} &\calL_2\ar[u]_{\nu_2}\ar[d]^{\nu'_2}\ar@{}[r]|{\beta_2\Downarrow} &\calH
\\
&\calK'\ar[ul]^{\upsilon'}\ar[ur]_{\varphi'}&&&&\calK'\ar[ul]^{\upsilon'}\ar[ur]_{\varphi'}
}
$$
represent the same 2-cell  when there are essential equivalences and 2-cells
as in the following:
\begin{equation}\label{d:equivreln}
\xymatrix@R=3em{
&\calK
\\
\calL_1 \ar[ur]^{\nu_1} \ar[dr]_{\nu_1'}
		&	\ar[l]_{\lambda_1} \calM \ar[r]^{\lambda_2} \ar@{}[u]|{\stackrel{\gamma}{\Leftarrow}}
			\ar@{}[d]|{\stackrel{\gamma'}{\Leftarrow}}
		&\calL_2 \ar[ul]_{\nu_2} \ar[dl]^{\nu'_2}
\\
&\calK'
}
\end{equation}
such that the composite of the pasting diagram
$$
\xymatrix@R=3em@C=3em{
&\calK\ar[drr]^\upsilon
\\
\calL_1 \ar[dr]_{\nu_1'}
		&	\ar[l]_{\lambda_1} \calM \ar[r]^{\lambda_2}
			\ar@{}[d]|{\stackrel{\gamma'}{\Leftarrow}}
		&\calL_2 \ar@{}[r]|{\stackrel{\alpha_2}{\Downarrow}}  \ar[ul]^{\nu_2} \ar[dl]_{\nu'_2} & \calG
\\
&\calK'\ar[urr]_{\upsilon'}
}
$$
is equal to the composite of the pasting diagram
$$
\xymatrix@R=3em@C=3em{
&\calK\ar[drr]^\upsilon
\\
\calL_2 \ar[ur]^{\nu_2}
		&	\ar[l]_{\lambda_2} \calM \ar[r]^{\lambda_1} \ar@{}[u]|{\stackrel{\gamma}{\Rightarrow}}
		&\calL_1 \ar@{}[r]|{\stackrel{\alpha_1}{\Downarrow}}  \ar[ul]^{\nu_1} \ar[dl]_{\nu'_1} & \calG
\\
&\calK'\ar[urr]_{\upsilon'}
}
$$
and the composite of the pasting diagram
$$
\xymatrix@R=3em@C=3em{
&\calK\ar[drr]^\varphi
\\
\calL_1  \ar[dr]_{\nu_1'}
		&	\ar[l]_{\lambda_1} \calM \ar[r]^{\lambda_2}  \ar@{}[d]|{\stackrel{\gamma'}{\Leftarrow}}
		&\calL_2 \ar@{}[r]|{\stackrel{\beta_2}{\Downarrow}} \ar[ul]^{\nu_2} \ar[dl]_{\nu'_2}& \calH
\\
&\calK'\ar[urr]_{\varphi'}
}
$$
is equal to the composite of the pasting diagram
$$
\xymatrix@R=3em@C=3em{
&\calK\ar[drr]^\varphi
\\
\calL_2 \ar[ur]^{\nu_2}
		&	\ar[l]_{\lambda_2} \calM \ar[r]^{\lambda_1} \ar@{}[u]|{\stackrel{\gamma}{\Rightarrow}}
		&\calL_1 \ar@{}[r]|{\stackrel{\beta_1}{\Downarrow}} \ar[ul]^{\nu_1}\ar[dl]_{\nu_1'}& \calH
\\
&\calK'\ar[urr]_{\varphi'}
}
$$
Intuitively, two 2-cell diagrams represent the same 2-cell when there is a further Morita equivalent
groupoid for the domain so that the original 2-cells agree on this further groupoid representation.   For instance, if  two paths into an orbigroupoid $\calH$ 
are defined using two distinct subdivisions of the unit interval, giving rise to orbigroupoids $\calI^3$ and $\calI^4$ say,
a 2-cell diagram could represent  a subdivision which is a common refinement of the two subdivisions given.
In this case,  two 2-cell diagrams would represent the same 2-cell if there is a further subdivision on which the 2-cells become the same.  

We define a $2$-category of orbispaces with objects the orbigroupoids, morphisms given by generalized maps and $2$-cells given by the equivalence classes of the 2-cell diagrams as described here.     Two generalized maps are considered to represent the same map of orbispaces if they have a 2-cell between them, but when we use these to create a mapping space which itself has the structure of an orbispace, we will need to retain information about the ways in which these generalized maps were identified.  We will illustrate this process in the next section.

\section{Mapping Spaces for Orbigroupoids: Defining  $\GMap(\calG, \calH)$}\label{S:GMap}

In this section, we construct a mapping groupoid $\GMap(\calG, \calH)$ between two orbigroupoids.   
This will be based on the definitions from Section \ref{S:hom}, with the maps defined by homomorphisms 
and identifications between maps given by natural transformations;  we explain how to topologize this 
structure so that we obtain a topological groupoid.    We will show examples where this again becomes an orbigroupoid, hence representing an orbispace structure on the mapping object.    Actually constructing the correct mapping space for any orbispaces using these mapping groupoids  takes some technical work which is not addressed here, but is detailed in the paper \cite{PS-tocome}.  In this paper, we merely illustrate how the orbispace structure comes into existence in these basic building blocks.

The topology we will put on our mapping groupoid is based  on topologies defined for various mapping spaces.  Topological spaces as a category are not Cartesian closed, and so various substitute categories are commonly used instead.  One  standard approach is to work with compactly generated spaces.  However, if we use the compact open topology to define a topology on a mapping space ${\rm Top}(X,Y)$, this may not be compactly generated even when $X$ and $Y$ are.    This is commonly fixed by applying the $k$-functor, which adds in some open and closed sets to the topology and makes $k{\rm Top}(X, Y)$ compactly generated again.  We will use this standard dodge in defining the topology on our orbigroupoids.  Therefore for any of the spaces that follow, $\Map(X,Y)$ will denote the topological space $k{\rm Top}(X,Y)$ obtained by taking the compact open topology and then applying the $k$-functor.  See \cite{steenrod, whitehead} for further details.

The mapping groupoid $\GMap(\calG, \calH)$  is defined as follows.

{\bf Objects}  The objects of $\GMap(\calG, \calH)_0$ are the  homomorphisms
 $\mathcal{G} \to \mathcal{H}$.
As defined in Section \ref{S:hom}, such a functor $f$ is defined by a map on objects $f_0:  \calG_0 \to \calH_0$ and a
map on arrows $f_1:\calG_1 \to \calH_1$.
We  topologize these  as a  subspace of $\Map(\calG_0,\calH_0)\times \Map(\calG_1, \calH_1)$
with the $k$-ified compact open topology.
The subspace taken will be those pairs of maps which make all the necessary diagrams commute:
$sf_1 = f_0s$, $tf_1 = f_0t$, $m(f_1 \times f_1) = f_1m$, and  $uf_0 = f_1u$.
So,
\begin{eqnarray*}
\GMap(\calG,\calH)_0&=&\{(f_0,f_1)\in\Map(\calG_0,\calH_0)\times \Map(\calG_1, \calH_1)|\\
&&sf_1 = f_0s, tf_1 = f_0t, m(f_1 \times f_1) = f_1m, uf_0 = f_1u\}
\end{eqnarray*}
This is a closed subspace of $\Map(\calG_0,\calH_0)\times \Map(\calG_1, \calH_1)$
Notice that in this subspace,
 $f_0$ is completely determined by $f_1$,
since we can view $\calG_0$ as a subspace of $\calG_1$ via the map to the identity arrows $u$.
So we can think of a point in this space as being defined by just  an $f_1$ which preserves composition and
maps $u(\calG_0)$ to $u(\calH_0)$, and  represent this same subspace as
$$\GMap(\calG,\calH)_0=\{f\in\Map(\calG_1,\calH_1)|\,m(f\times f) = fm,\,f(u(\calG_0))\subseteq u(\calH_0)\},$$
equipped with the subspace topology.  In this representation, we can see that this is a closed supbspace because 
 $\calH$ is \'etale.

{\bf Arrows}  The arrows of the mapping groupoid $\GMap(\calG, \calH)_1$  are defined
 by natural transformations between the functors which are its objects.
Explicitly, a natural transformation $\alpha\colon f \to f'$  is given by
$\alpha \in \Map(\calG_0, \calH_1)$  making the following diagram commute:
for $g\colon  x\to y $ in $\calG_1$,
 \begin{equation}\label{natlsq}
\xymatrix{
f_0(x) \ar[r]^{\alpha(x)} \ar[d]_{f_1(g)}   & f'_0(x) \ar[d]^{f'_1(g)} \\
f_0(y) \ar[r]_{\alpha(y)} & f'_0(y)}
\end{equation}
We can think of this as a triple $(f,\alpha,f')$ 
where  $f = (f_0, f_1)$ and $f' = (f'_0, f'_1)$ satisfy the functor relations,
and in addition,  $(f, \alpha, f')$ satisfies  $m(\alpha(x), f'_1(g)) = m(f_1(g), \alpha(y))$ (and this equation makes sense, i.e.
$s\alpha=f_0$ and $t\alpha=f'_0$).  So we are looking at a subspace of the space
\begin{eqnarray}
[\Map(\calG_0, \calH_0) \times \Map(\calG_1, \calH_1)]
&\times _{\Map(\calG_0, \calH_0) }& \Map(\calG_0, \calH_1) \notag\\
&\times_{\Map(\calG_0, \calH_0) } &
[ \Map(\calG_0, \calH_0) \times \Map(\calG_1, \calH_1) \label{bigspace}]
\end{eqnarray}
 with the $k$-ified compact open topology.

Now we look closer at the space of arrows of  $\GMap(\calG, \calH)$.
Explicitly, the elements of the subspace of natural transformations  can be written
as $(f_0, f_1, \alpha, f'_0, f'_1)$ such that the following holds:
 \begin{itemize}
 \item $s \alpha = f_0$ and $t \alpha = f'_0$
 \item  $m(\alpha s, f'_1) = m(f_1, \alpha t)$
 \item $sf_1 = f_0s, tf_1 = f_0t, m(f_1 \times f_1) = f_1m,  uf_0 = f_1u$    (i.e. $f$ is a functor)
  \item $sf'_1 = f'_0s, tf'_1 = f'_0t, m(f'_1 \times f'_1) = f'_1m, uf'_0 = f'_1u$    (i.e. $f'$ is a functor)
 \end{itemize}
 Note that $\GMap(\calG,\calH)_1$ is a closed subspace of (\ref{bigspace}).
Again, we notice that parts of such 5-tuples completely determine the rest, just as was the case for the space of objects.

{\bf Structure Maps} Now we look at the structure maps of the groupoid $\GMap(\calG, \calH)$.
\begin{itemize}
\item The {\em source map} $s\colon \GMap(\calG, \calH)_1 \to \GMap(\calG, \calH)_0$ is defined by
$$s(f_0, f_1, \alpha, f_0', f_1') = (f_0, f_1).$$
\item Similarly, the {\em target map} $t\colon \GMap(\calG, \calH)_1 \to \GMap(\calG, \calH)_0$ is defined by
$$t(f_0, f_1, \alpha, f_0', f_1') = (f'_0, f'_1).$$
\item
 The {\em unit map} $u: \GMap(\calG, \calH)_0 \to \GMap(\calG, \calH)_1$  is defined by
$$u(f_0, f_1) = (f_0, f_1, u, f_0, f_1).$$
\item The {\em composition map}
$$m\colon \GMap(\calG, \calH)_1 \times_{\GMap(\calG, \calH)_0} \GMap(\calG, \calH)_1 \to  \GMap(\calG, \calH)_1$$
is defined by $$m((f_0, f_1, \alpha , f_0', f_1') , (f_0', f_1', \beta, f_0'', f_1'')) = (f_0, f_1, m(\alpha, \beta), f_0'', f_1'').$$
\item The {\em inverse map} $i\colon \GMap(\calG,\calH)_1\to\GMap(\calG,\calH)_1$ is defined by
$$i(f_0, f_1, \alpha, f_0', f_1') = (f'_0, f'_1, i\alpha, f_0,  f_1).$$
\end{itemize}
It is clear from the definitions that each of these will be continuous,
since they are defined from continuous maps on the various components.
Hence,  $\GMap(\calG, \calH)$ is a topological groupoid.

This groupoid may not be an orbigroupoid, but in many cases it is.   So we have a way of defining a mapping orbispace which  carries the isotropy structure on the maps.  This may not be the correct mapping space in the 2-category of orbispaces because of the need to consider generalized maps.  However, the basic orbispace structure carried by these  mapping groupoids $\GMap(\calG, \calH)$ is what is used to create the mapping orbispace of \cite{PS-tocome}.    Therefore we finish this section by carefully working through several examples.  

The first example comes from the following observation.  

\begin{prop} \label{p:e-to-G}
There is an isomorphism of topological groupoids $$\GMap(*_1,\G)\cong\G.$$ \end{prop}

This result follows from the fact that the category of topological groupoids in compactly generated spaces is Cartesian closed,
but we opt to include a proof to illustrate the process of creating these mapping groupoids, and as a  first  example of the more general case of calculating $\GMap(*_G,\calG)$ where $G$ is an arbitrary finite group.

\begin{proof}
The objects of $\GMap(*_1,\calG)_0$ are homomorphisms $f:*_1\to\calG$ 
defined on the objects, $f_0:*\to\calG_0$, and on the arrows, $f_1:1\to\calG_1$.  
Now $f_0$ can map $*$ into any point $y\in\calG_0$, and this determines the map on the arrow, $f_1(1)=id_y$.  
Thus the objects $f$ in $\GMap(*_1,\calG)$ are in one-to-one correspondence with the objects $y$ of $\calG$.  
Moreover, the topology is defined using the subspace topology on the mapping space 
$\Map(*, \calG_0) \times \Map(1, \calG_1) $ with $uf_0 = f_1u$, which means that the correspondence 
$f \to y$ is a homeomorphism.  

The arrows of $\GMap(*_1,\calG)_1$ consist of natural transformations between homomorphisms.   If $f, f'\in\GMap(*_1,\calG)_0$, a natural transformation $\alpha:f\Rightarrow f'$ is given by $\alpha(*)=g\in\calG_1$ such that $s\circ\alpha(*)=y=f_0(*)$, $t\circ\alpha(*)=y'=f'_0(*)$, and the following diagram commutes:
\[ \xymatrix{y\ar[r]^{g}\ar[d]_{f_1(1)=id_y} & y'\ar[d]^{f'_1(1)=id_{y'}}\\
y\ar[r]_{g} & y'
} \]
Since this diagram always commutes, the existence of $\alpha$ depends only on the existence of an arrow $g:f_0(*)\to f'_0(*)$ in $\calG_1$.  Hence,  the arrows in $\GMap(*_1,\calG)$ are in one-to-one correspondence with the arrows in $\calG$.  Again, looking at the topology of $\GMap(*_1,\calG)_1$, it is easy to see that   this correspondence is a homeomorphism.  
\end{proof}

In the next examples, we will consider the mapping spaces from the one-point space $*_G$ (with isotropy group $G$) to the triangular billiard $\calTB$ of  Example \ref{billiard}.
We will compare the results for isotropy groups  $G$ taken to be the trivial group, $\mathbb{Z}/2$, $\mathbb{Z}/3$, and $\mathbb{Z}/6$.  

In Proposition \ref{p:e-to-G} we proved that for any orbigroupoid,  $\GMap(*_1, \calG) \simeq \calG$.  Thus $\GMap(*_1, \calTB) \simeq \calTB$.

\begin{eg} \label{example:maps-from-a-2point}
Here we describe  $\GMap(*_{\mathbb{Z}/2}, \calTB)$.   We denote the elements of $\mathbb{Z}/2$  by $1$ and $\tau$, and use the same notation for  the corresponding arrows in $*_{\mathbb{Z}/2}$.    We will denote the component disks of the object space  $\calTB_0$ by disks $A$, $B$, and $C$.  
Recall that the arrow space $\calTB_1$ includes six disjoint copies of each disk, one for each element of $D_3$.  We will denote the arrow disk  that represents the rotation element $\rho$ of $D_3$ acting on disk $A$ by $\rho_A$.  So arrows in the disk $\rho_A$ have source $x \in A$ and target $\rho x \in A$.     The other arrow disks will be denoted similarly.  The arrow space $\calTB_1$ also contains overlaps representing the arrows with source in one disk and target in another.

The object space of $\GMap(*_{\mathbb{Z}/2}, \calTB)$ consists of homomorphisms from $*_{\mathbb{Z}/2}$ to $\calTB$.  Any such homomorphism will map the object $*$ to a point $y$ in one of the disks $A$, $B$, or $C$, and the arrow $1$ to the corresponding point in $1_A$, $1_B$, or $1_C$.  The arrow $\tau$ must map to an arrow $y \to y$, so the choices for the image of the arrow $\tau$ will depend on the point $y$.  We will consider what happens when $y$ is in disk $A$; the cases with $y$  in disks $B$ and $C$ are similar.  If  $f_0(*)=y$ has trivial isotropy, so $y$ is neither the center point nor a point on one of the reflections lines in disk $A$, then  $f_1(\tau)$ must map to the point representing $id_y$ in  disk $1_A$;  see map  $f$ in Figure \ref{fig:func}.
If $f'_0(*)=y$ is a point on one of the reflection lines, say the one for the reflection $\sigma$ on disk $A$, but  not the center point, then $f'_1(\tau)$ can be either in disk $1_A$ or disk $\sigma_A$; see map $f'$ in Figure \ref{fig:func}.  Similarly for $y$ fixed by  $\sigma\rho$ or $\sigma\rho^2$.   
Finally, if $f''_0(*)=y$ is the center point of disk A, then every element of $D_3$ fixes $y$.   However, since $\tau$ has order $2$, $f''_1(\tau)$ must  be the center point of one of the arrow disks $1_A$, $\sigma_A$, $\sigma\rho_A$, or $\sigma\rho^2_A$;  see map $f''$ in  Figure \ref{fig:func}.
\begin{figure}[h]
\includegraphics{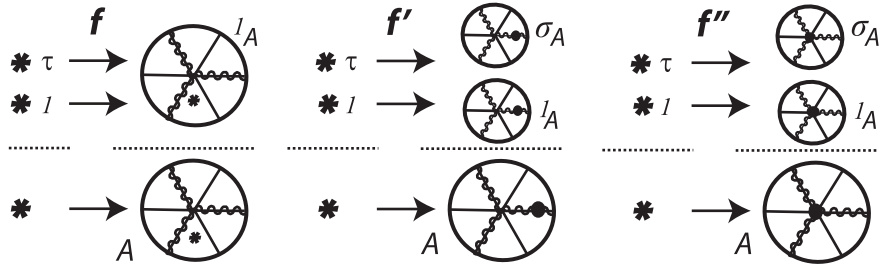}
\caption{Some objects in $\GMap(*_{{\mathbb Z}/2}, \mathcal{TB})_0$ }%
\label{fig:func}
\end{figure}

As a topological space, $\GMap(*_{\mathbb{Z}/2}, \calTB)_0$  is defined as a subspace of  $\Map(*, \calTB_0) \times \Map({\mathbb{Z}/2}, \calTB_1) $  with the compatibility conditions making $f_0$ and $f_1$ into a functor of groupoids.    So if $f_1(\tau) $ and $f'_1(\tau)$  land in a different components of $\calTB_1$ as in the above case, then $f $ and $f'$ will be in different components of  $\GMap(*_{\mathbb{Z}/2}, \calTB)_0$.     On the other hand, the example homomorphisms $f'$ and $f''$ from above are in the same component, the component which takes $y$ to the points fixed by $\sigma$ and takes $\tau$ to the arrow disk $\sigma_A$,  and this component is  homeomorphic to the line segment that is the subspace of disk $A$ fixed by $\sigma$.  
Similarly we get line segments corresponding to $f_1(\tau) = \sigma \rho $ and $f_1(\tau) = \sigma \rho^2 $.  So each disk $A$, $B$ and $C$ contributes one disk and three line segments to the space of objects  $\GMap(*_{\mathbb{Z}/2}, \calTB)_0$.     Note that there are $4$ homomorphisms with image $y = c$ the  center point, showing up in each of the $4$ components.  See Figure \ref{fig:ob}.  
\begin{figure}[h]
\includegraphics{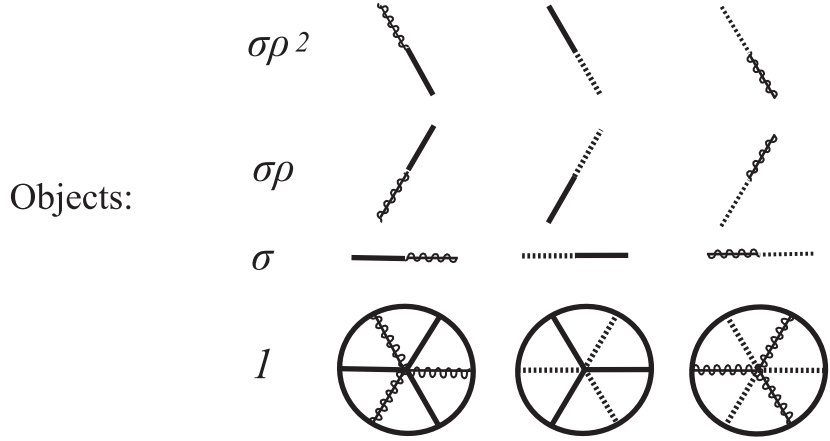}
\caption{Space of objects $\GMap(*_{{\mathbb Z}/2}, \mathcal{TB})_0$ }
\label{fig:ob}
\end{figure}

Now we consider the space of arrows $\GMap(*_{{\mathbb Z}/2}, \calTB)_1 $,
given by  natural transformations $\alpha\colon f \to f'$ between functors $f$ and $f'$.  These are  defined by a map
$\alpha \in \Map(*, \calTB_1)$ with the properties that $\alpha(*):f_0(*)\to f'_0(*)$, and for any map $g:*\to * $ in the arrow space of $*_{\mathbb{Z}/2}$, the following diagram commutes:
 \begin{equation*}
\xymatrix{
f_0(*) \ar[r]^{\alpha(*)} \ar[d]_{f_1(g)}   & f'_0(*) \ar[d]^{f'_1(g)} \\
f_0(*) \ar[r]_{\alpha(*)} & f'_0(*)}
\end{equation*}
Note that for commutativity to hold, we must have  the conjugacy relation
\begin{equation}\label{eq:Z2rel}
\alpha(*)f_1(g)\alpha(*)^{-1} = f'_1(g).
\end{equation}
If  $g=1$, then $f_1(g)$ and $f'_1(g)$ are identity maps and the diagram always commutes.  So 
suppose $g=\tau$.    If $f_1(\tau)$ is an identity, then the conjugacy relation becomes $\alpha(*)\alpha(*)^{-1} = f'_1(\tau)$, and so $f'_1(\tau)$ must also be an identity.  In this case,  the natural transformations $\alpha:f\Rightarrow f'$ correspond exactly to  maps $f_0(*)\to f'_0(*)$ in $\calTB_1$.  Thus these arrows  form a subspace $\GMap(*_{{\mathbb Z}/2}, \calTB)_1$ consisting of a copy of $\calTB_1$.
This is not surprising since these homomorphisms factor through $*_1$, and by Proposition \ref{p:e-to-G}, homomorphisms from $*_1$ and their natural transformations form a copy of  $\calTB$.  

Now, consider the case where  $f_1(\tau)$ is a non-identity arrow.  The conjugacy relation~ (\ref{eq:Z2rel}) implies that $f'_1(\tau)$  must also be a non-identity arrow.  Let $y=f_0(*)$ and $y'=f'_0(*)$.  By our analysis of $\GMap(*_{{\mathbb Z}/2}, \mathcal{TB})_0$, we know that $y$ and $y'$ are points on reflection lines in disks $A$, $B$, or $C$ (they may be in different disks).  

First, suppose that $y$ and $y'$ are in the same disk, say disk $A$, and  $f_1(\tau)$ is the corresponding arrow in one of the disks $\sigma_A$, $\sigma\rho_A$, or $\sigma\rho^2_A$. 
Note that $\alpha(*):y\to y'$, so  $y'$ must be one of $y$, $\rho y$, or $\rho^2 y$, and  $\alpha f_1(\tau)\alpha^{-1} = f'(\tau)$.  So for example, if  $y'= \rho y$ and $f(\tau) = \sigma$, there are two possible choices for $\alpha$ that will satisfy  (with abuse of notation) $\alpha\sigma\alpha^{-1} = \sigma\rho$:  $\alpha(*)$ must be the  arrow corresponding to $y$ in either disk $\rho_A$ or disk $\sigma\rho^2_A$.  Similarly, for each such pair $y, y'$ we have exactly  two natural transformations from $f$ to $f'$.  Note that this includes the case when $f=f'$.    Again, the topology is based on a subspace of the  mapping space $\Map (*,  \calTB _1)$, which means that we have components corresponding to the $\alpha$'s that map $\tau$ into the various components of $\calTB_1$.     So the arrow space  $\GMap(*_{{\mathbb Z}/2}, \calTB)_1$ contains $6$ copies of each reflection line, with two copies mapping to each of itself and  the other $2$ lines.  

Now consider the case where $y$ and $y'$ are in different disks, say $y$ is in disk $A$ and $y'$ is in disk $C$. Then if  $\alpha(*):y\to y'$, both $y$ and $y'$ must be in the overlap area of disks $A$ and $C$.  In this case, there will be exactly two such $\alpha(*)$ arrows glueing an end of the line segment associated with disk $A$ to an end of the line segments associated with disk $C$, one with a reflection and one without (see Example \ref{billiard});    for each of these choices, the conjugacy relation \ref{eq:Z2rel} will hold.     Note that topologically,  these are subspaces of the $3 \times 36=108$ components of the overlap arrows in $\calTB_1$, now arranged to glue together the overlaps of the $3 \times 3 = 9$ line segments, with $2$ glueings in each direction for each choice of pairs that overlap.   We illustrate the effect of all of the arrows (glueings) on the line segments in $\GMap(*_{\mathbb{Z}/2},\calTB)_0$ in Figure \ref{fig:nat-trans-ident};  note that we have not drawn the multiplicity of the glueing arrows.  
\begin{figure}[h]
\includegraphics{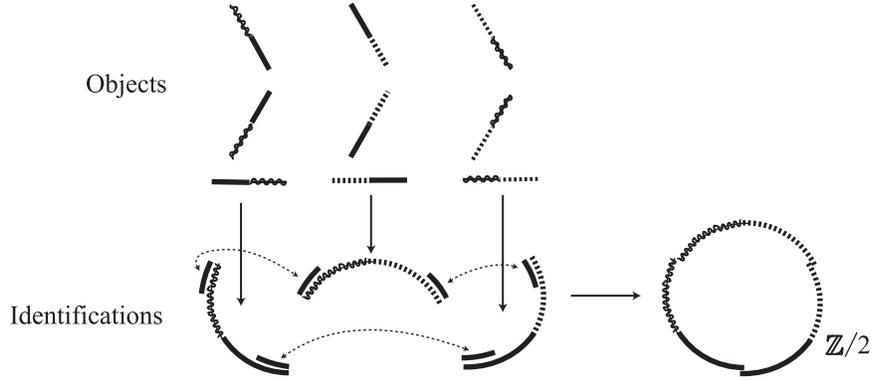}
\caption{Glueings of the line segment components of $\GMap(*_{{\mathbb Z}/2}, \mathcal{TB})_0$ }
\label{fig:nat-trans-ident}
\end{figure}
As shown in the figure, the quotient space is a circle with $\mathbb{Z}/2$-isotropy at each point.

Putting all of the natural transformations together, we have that the arrow space $\GMap(*_{\mathbb{Z}/2},\calTB)_1$  contains a copy of $\calTB_1$, 6 line segments for each of the 3 reflection lines in each of the disks $ A$, $B$, and $C$, and one shorter line segment for each overlap in $\calTB_1$.   The line segments of $\GMap(*_{\mathbb{Z}/2},\calTB)_0$ wind up glued together into a circle with $\mathbb{Z}/2$-isotropy at each point, and the three disks in $\GMap(*_{\mathbb{Z}/2},\calTB)_0$ are glued together into a copy of $\calTB$, as in  Figure \ref{fig:circle-bill}.     \begin{figure}[h]
\includegraphics{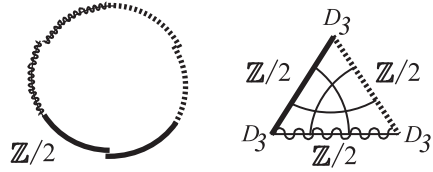}
\caption{The quotient space of $\GMap(*_{{\mathbb Z}/2}, \mathcal{TB})$ }
\label{fig:circle-bill}
\end{figure}
\end{eg}

This example shows how the mapping groupoid inherits both topology and isotropy structure from the domain and codomain groupoids, giving a rather interesting structure.  For comparison, we will 
now construct $\GMap(*_{{\mathbb Z}/3}, \mathcal{TB})$, which we will see is a lot simpler since there is only one point 
with an isotropy group whose order is a multiple of $3$.

\begin{eg} \label{example:maps-from-a-3point}
Now we consider $\GMap(*_{\mathbb{Z}/3},\calTB)$.  
We denote the elements of $\mathbb{Z}/3$ by $1$, $\nu$, and $\nu^2$, and hence also  the arrows of $*_{\mathbb{Z}/3}$.  
The analysis of  $\GMap(*_{\mathbb{Z}/3},\calTB)$ is very similar to that of Example \ref{example:maps-from-a-2point}.    The difference is that $\mathbb{Z}/3$ has elements of order 3 instead of order 2, and so we will be looking at  points in $\calTB_0$ with  order $3$ isotropy elements.  
As before, any homomorphism $f\in\GMap(*_{\mathbb{Z}/3},\calTB)_0$ will map the object $*$ to a point $y$ in one of the disks $A$, $B$, or $C$, and the arrow $1$ to the corresponding point in $1_A$, $1_B$, or $1_C$.  If $f_0(*)=y$ does not have order 3 isotropy, i.e., if $y$ is not the center point of a disk, then $f_1(\nu)$ and $f_1(\nu^2)$ must also map to the corresponding identity point in $1_A$, $1_B$, or $1_C$.  If $f_0(*)=y$ is the center point of a disk, say disk $A$, then $f_1(\nu)$ can be any arrow of order $3$, namely the center point of one of the disks $1_A$, $\rho_A$, or $\rho^2_A$.  Since $\nu^2$ is the inverse of $\nu$, $f_1(\nu^2)$ is  determined by $f_1(\nu)$ and must be  the center point of $1_A$, $\rho^2_A$, or $\rho_A$, respectively.   So we get 3 homomorphisms with $f_0(*)=y$ the center of disk $A$, one of which is the centre point for an identity component disk, and two isolated points which are their own components.  The other two disks are similar, and each disk $A$, $B$ and $C$ contributes one disk and two points to $\GMap(*_{\mathbb{Z}/3},\calTB)_0$.  See Figure \ref{fig:ob3}.

\begin{figure}[h]
\includegraphics{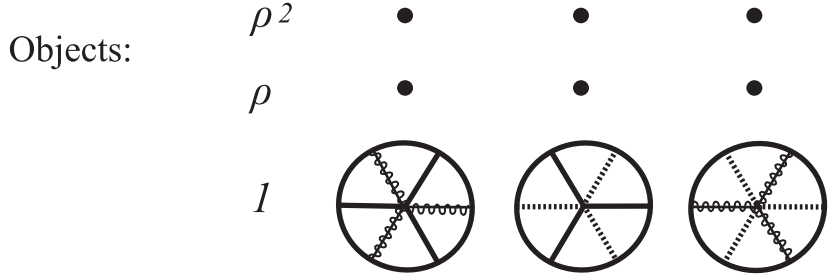}
\caption{Space of objects $\GMap(*_{{\mathbb Z}/3}, \mathcal{TB})_0$ }
\label{fig:ob3}
\end{figure}

For the space of arrows $\GMap(*_{\mathbb{Z}/3},\calTB)_1$, consider the natural transformations $\alpha$ between functors  $f,f'\in \GMap(*_{\mathbb{Z}/3},\calTB)_0$.  As before, $\alpha\in\Map(*,\calTB_1)$  is such that $\alpha(*):f_0(*)\to f'_0(*)$, and for any  arrow  $g $ of  $ *_{\mathbb{Z}/3}$,   $\alpha(*)f_1(g)\alpha(*)^{-1} = f'_1(g)$.   The conjugacy relation always holds for $g=1$, and if it holds for $g=\nu$, then it will also hold for $g=\nu^2$, so we only need to check that the conjugacy relation holds for $g=\nu$.

Suppose $g=\nu$.    If $f_1(\nu)$ is an identity, then the conjugacy relation becomes $\alpha(*)\alpha(*)^{-1} = f'_1(\nu)$, and $f'_1(\nu)$ must also be an identity, and 
As in the previous example,  the natural transformations $\alpha:f\Rightarrow f'$ correspond exactly to  maps $f_0(*)\to f'_0(*)$ in $\calTB_1$, and these arrows form a subspace $\GMap(*_{{\mathbb Z}/3}, \calTB)_1$ consisting of a copy of $\calTB_1$.

Now, consider the case where  $f_1(\nu)$  and $f'_1(\nu)$  are non-identity arrows.  Let $y=f_0(*)$ and $y'=f'_0(*)$;   then $y$ and $y'$ are center points of disks $A$, $B$, or $C$.   They must both be in the same disk for there to be an arrow $\alpha(*):y\to y'$ in $\calTB_1$.
If $y=y'$ is the center point of disk $A$, there are three possible $\alpha$'s for each choice of $f_1$  and   $f_1'$.   For example, if $f_1(\nu)\in\rho_A$ and $f'_1(\nu)\in\rho^2_A$, then $\alpha(*)$ is the center point of any one of $\sigma_A$, $\sigma\rho_A$ and $\sigma\rho^2_A$.   These arrows glue together the pairs of object points in $\GMap(*_{\mathbb{Z}/3},\calTB)_0$  for disk $A$, resulting in one point with $\mathbb{Z}/3$-isotropy in the quotient space.
The result is similar if $y$ is the center point of disk $B$ or $C$.

All together, we have that the arrow space $\GMap(*_{\mathbb{Z}/3},\calTB)_1$  contains a copy of $\calTB_1$ and 12 disjoint arrow points  associated to each of the disks $A$, $B$, and $C$. 
The quotient space for the mapping groupoid $\GMap(*_{\mathbb{Z}/3},\calTB)$ is composed of a copy of $\calTB$ and three disjoint points, each with $\mathbb{Z}/3$-isotropy.   See Figure \ref{fig:3p-bill}.

\begin{figure}[h]
\includegraphics{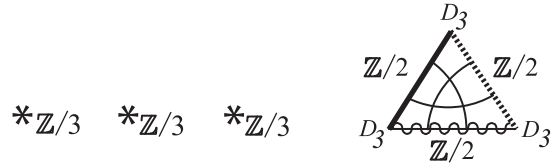}
\caption{The quotient space of $\GMap(*_{{\mathbb Z}/3}, \mathcal{TB})$}
\label{fig:3p-bill}
\end{figure}
\end{eg}

We have seen how maps from $*_{\mathbb{Z}/2}$ pick up information about isotropy of order $2$, and maps from $*_{\mathbb{Z}/3}$ pick up information about isotropy of order $3$.  If we want to get all of it combined, we consider $G= \mathbb{Z}/6$. 

\begin{eg} \label{example:maps-from-a-6point}
Finally we consider
$\GMap(*_{\mathbb{Z}/6},\calTB)$, generated by $\gamma$ with $\gamma^6 =  1$.  Note that  $\mathbb{Z}/6$ is isomorphic to
$\mathbb{Z}/2\times\mathbb{Z}/3$, where the order two generator $\tau = \gamma^3$ and the order three generator $\nu = \gamma^2$.  Again, we use the same notation for the arrows of 
$*_{\mathbb{Z}/6}$.   The analysis is a
combination of the analyses above for $*_{\mathbb{Z}/2}$ and
$*_{\mathbb{Z}/3}$.

Any homomorphism $f\in\GMap(*_{\mathbb{Z}/6},\calTB)_0$ maps $*$ to a
point $y$ in one of the disks of $A$, $B$, or $C$ and arrow $1$ to
the corresponding point in $1_A$, $1_B$, or $1_C$.  We have to consider 
points with isotropy of order $2$ or $3$ (or trivial isotropy).  
Since a homomorphism respects composition
and $\gamma$ generates $\mathbb{Z}/6$,  the image $f_1(\gamma)$ determines the rest of $f_1$. 

If $f_0(*)=y$ does
not have isotropy of order $2$ or $3$, then $\gamma$  (and hence all of
the other arrows in $\mathbb{Z}/6$) maps to $id_y$, and we get a homeomorphic copy of each of the disks.  If $y$ has order $2$ isotropy (and hence is on a reflection
line) then $f_1$ can map $\gamma$ to any of the arrows determined in the 
$*_{\mathbb{Z}/2}$ case, that is, to reflection arrows.  If, for example, $f_1(\gamma) $ is in the $ \sigma$ arrow disk,  then $f_1(\gamma^3)=  f_1(\gamma^5) $ are also the arrow corresponding to $y $ in the $\sigma$ disk,  and $f_1(\gamma^2) $ and $f_1(\gamma^4)$  map to the identity arrow on $y$.    As before,  this  
contributes three line segments for each disk to the object space.   If $y$ has order $3$ isotropy (and hence is a 
center point) then $f_1(\gamma)$ can be any of the order $3$ rotation arrows.   So we also have two isolated points for each disk. 

Overall,  we get a disk, three line segments, and two points in $\GMap(*_{\mathbb{Z}/6},\calTB)_0$
associated to each disk. See Figure \ref{fig:ob6}.

\begin{figure}[h] \includegraphics{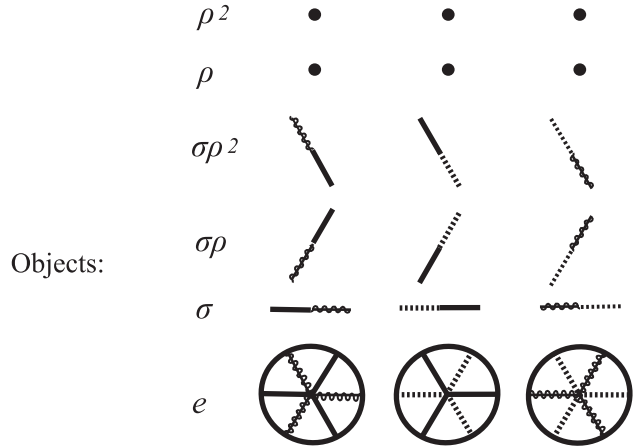}
\caption{Space of objects $\GMap(*_{{\mathbb Z}/6}, \mathcal{TB})_0$ }
\label{fig:ob6} \end{figure}

The space of arrows  $\GMap(*_{\mathbb{Z}/6},\calTB)_1$ also correspond to the arrows considered in Examples \ref{example:maps-from-a-2point} and \ref{example:maps-from-a-3point} , since they are based on conjugacy relations of arrows in $\calTB$.  
 We get 6 line segments in $\GMap(*_{\mathbb{Z}/6},\calTB)_1$  
 for each reflection line in each of the disks in $\calTB_0$, and 
12 disjoint arrow points  in $\GMap(*_{\mathbb{Z}/6},\calTB)_1$ for the isolated points of each of the disks.   The quotient space of the mapping groupoid
$\GMap(*_{\mathbb{Z}/6},\calTB)$ is given in Figure
\ref{fig:3p-circle-bill}.  Notice that it includes each of the
components of the quotient spaces of $\GMap(*_{\mathbb{Z}/2},\calTB)$
and $\GMap(*_{\mathbb{Z}/3},\calTB)$, as expected.

\begin{figure}[h] \includegraphics{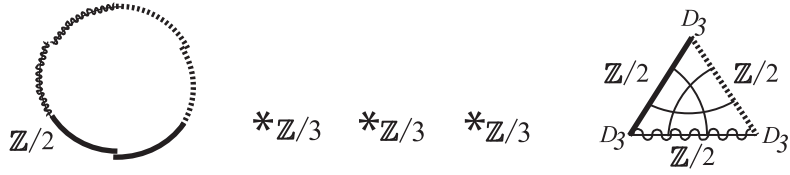}
\caption{The quotient space of $\GMap(*_{{\mathbb Z}/6}, \mathcal{TB})$ }
\label{fig:3p-circle-bill} 
\end{figure}

\end{eg}

\begin{eg} \label{e:maptocone}
We now describe the mapping space   $\GMap(\calI ,\calC_3)$ from the interval to the order three cone point of Example \ref{cone}.
 For any $f\in\GMap(\calI ,\calC_3)_0$, we must have $f_0:\calI_0\to D^2$ and $f_1:\calI_1\to (\calC_3)_1=\mathbb{Z}/3\times D^2$, compatible with the functor conditions.    Any  continuous map from the interval to the disk will work for $f_0$.   Then, to be compatible,  $f_1$ must be of the form $f_1(id_x)=(id_{f_0(x)} ,f_0(x))$ for each $x\in \calI_0$.  Thus, $\GMap(\calI,\calC_3)_0$ is homeomorphic to ${\rm Map}(I,D^2)$,  the space of continuous maps $I\to D^2$ with the $k$-ified compact open topology.

  Next, to understand $\GMap(\calI,\calC_3)_1$, let $f\in \GMap(\calI, \calC_3)_0$.  We will characterize the natural transformations $\alpha$ by looking at the possible homomorphisms $f'\in \GMap(\calI, \calC_3)_0$ for which there is a natural transformation $\alpha :f\to f'$.   By the definition of a natural transformation, for each $x\in \calI_0$, we must have $\alpha(x)\in(\calC_3)_1$ such that $\alpha(x):f_0(x)\to f'_0(x)$.  It follows that $f'_0(x)$ must be in the orbit of $f_0(x)$ under the action of $\mathbb{Z}/3$.  Additionally, since $\alpha$ must be continuous,  it lands in one component of the arrows of $\calC_3$, and so $f'_0$ must be $1f$,  $\nu f$ or $\nu^2 f$ for all $x \in \calI_0$.  For each such $f'$, we get exactly one $\alpha :f\to f'$ where $\alpha$ maps into the component of the appropriate rotation.    The topology is given by a subspace of ${\rm Map}(I, D^2 ) \times {\rm Map}(I, \mathbb{Z}/3 \times D^2) $, and the compatibility in this case means we are really just looking at $ {\rm Map}(I, \mathbb{Z}/3 \times D^2)$ which is homemorphic to   $\mathbb{Z}/3 \times {\rm Map}(I, D^2)$, with source and target maps defined as in a translation groupoid.     
  
So the mapping groupoid is just the translation groupoid  $\GMap(\calI ,\calC_3) \simeq   \mathbb{Z}/3 \ltimes {\rm Map}(I, D^2)$.  
Note that each point in $\GMap(\calI,\calC_3)_0={\rm Map}(I, D^2)$ has trivial isotropy, except the point given by the map taking 
the entire interval $I$ to the cone point at the center of $D^2$, which has $\mathbb{Z}/3$-isotropy.

\end{eg}

This example can be generalized as follows.  

\begin{prop}

If $X$ is a connected orbispace with trivial isotropy,  represented as an orbigroupoid $X$ with object space $X$ and arrow space $X$ (the identity maps),  and $G$ is a finite group acting on a topological space $Y$, then $$\GMap(X, G \ltimes Y ) \simeq G \ltimes {\rm Map}(X, Y).$$  \end{prop}

\begin{proof}The argument is the same as given in Example \ref{e:maptocone} above.  Let $f\in \GMap(X, G\ltimes Y)_0$.   Then $f$ is  determined by $f_0:X\to Y$ since $X_1$ consists only of identity arrows.  So $\GMap(X, G\ltimes Y)_0$ is homeomorphic to ${\rm Map}(X,Y)$.    Arrows $\alpha:f\to f'$ in $\GMap(X, G \ltimes Y )_1$ are given by continuous maps $\alpha:X\to G\times Y$ such that for each $x\in X$, $\alpha(x):f_0(x)\to f'_0(x)$.  Thus, such an arrow exists only when  $f'_0(x)$ is in the orbit of $f_0(x)$ under the action of $G$, and continuity requires that $\alpha$ land in one component of the arrows  $G\times Y$.  Therefore   $f'_0$ is of the form $gf_0$ for some $g\in G$, and $\alpha$ is completely determined by $f$ and $g$;  so $\GMap(X, G\ltimes Y)_1$  is homeomorphic to $ G\times {\rm Map}(X,Y)$, with source and target maps as given by the translation groupoid  $G \ltimes {\rm Map}(X, Y).$
\end{proof}

In each of these cases, we produced an orbispace.  However, $\GMap(\calG,\calH)$ is not necessarily an orbigroupoid (i.e., \'etale and proper) in general.  Even when it is, we have examples showing that it is not invariant under Morita equivalence in the sense that 
when $\calG$ and $\calG'$ are Morita equivalent orbigroupoids, and $\calH$ and $\calH'$
are Morita equivalent orbigroupoids, $\GMap(\calG,\calH)$ and $\GMap(\calG',\calH')$ are not necessarily Morita equivalent;  see Examples \ref{ItoT} and \ref{e:codomain}.      Therefore the mapping groupoids we have created are not sufficient to be mapping objects for the category of orbispaces.    However, they do form the basic foundation for defining a mapping orbispace which encodes the generalized maps;  see  \cite{PS-tocome}.

 \section{Inertia Groupoids and Mapping Groupoids}\label{inertia}

In this section we will show that the inertia groupoid of \cite{kawasaki} can be obtained as a mapping groupoid.
\begin{dfn}  Given an orbigroupoid  $\calG$, its {\em inertia groupoid} $\wedge \calG$ is defined as follows:

\begin{itemize}
\item \textbf{Objects} $(\wedge \calG)_0$: objects in $(\wedge \calG)_0$ are loops,
 \begin{align*}
(\wedge \calG)_0 & = \{ g \in \calG_1 \vert s(g)=t(g)   \} \\
                       & = \{ (x, g) \in \calG_0 \times \calG_1 \vert x = s(g)=t(g) \}
\end{align*}
Note that since $s(g)=t(g)$, $g$ is an element of the isotropy group of $x=s(g)$,
i.e. $g \in G_x $

\item \textbf{Arrows} $(\wedge \calG)_1$: Let $x, y \in \calG_0$, and $h: x \to y $
be an arrow $h \in \calG_1$.  Then $h$ defines arrows   $(x, g) \to (h(x)=y, hgh^{-1})$ for all $g \in G_x$.
\end{itemize}
\end{dfn}

The examples we have seen thus far indicate that we might be able to describe the inertia groupoid $\Lambda\calG$
as 
a mapping groupoid of the form $\GMap(*_G,\calG)$. This is indeed the case, but if we want to do it as a map from $*_G$ with finite isotropy group $G$ (so that $*_G$ is an orbigroupoid) we need to require that  there is an upper bound 
on the orders of the isotropy groups of $\calG$. A natural way to ensure this is to require that $\calG$ is orbit compact.

\begin{dfn} 
An orbigroupoid $\calG$ is {\em orbit compact} if the quotient space $\calG_0/\calG_1$ is compact.
\end{dfn}




We will prove that for orbit compact groupoids, the inertia groupoid can be obtained as a mapping space $ \GMap(*_{{\mathbb Z}/n}),\calG)$.

\begin{rmk}
This condition is rather natural. In \cite{PS-tocome} we  require the condition of orbit compactness on $\calG$ 
in order to show that the  mapping groupoid  produced there gives an orbispace.  
This condition is equivalent to Haefliger's condition \cite{Hae} that $\calG_0$ have a relatively compact 
open subset which meets every orbit (used there in proving that the 2-category of \'etale groupoids and 
Hilsum-Skandalis maps is Cartesian closed.)
\end{rmk}

\begin{prop}
There is a functor $\Phi :\GMap(*_{{\mathbb Z}/n}, \calG) \to \Lambda(\calG)$.
\end{prop}

\begin{proof}
 Let ${\mathbb Z}/n $ be generated by an element $\sigma$ of order $n$.
We define the functor
 $\Phi$  as follows.
Objects of $\GMap(*_{{\mathbb Z}/n}, \calG)$ are defined by functors, so
consider a functor $ f:  *_{{\mathbb Z}/n} \to \calG$.   
Then  $f_0(*)=x \in \calG_0$ and $f_1(\sigma) = g \in \calG_1$ such that $s(g) = t(g) = x$.  So define
 $$ \Phi(f) = (x, g) = (f(*), f(\sigma)) \in (\Lambda G) _0$$

Arrows of $\GMap(*_{{\mathbb Z}/n}, \calG)$ are defined by natural transformations between functors.  
So  suppose that  $\alpha:  f \Rightarrow f'$ is a natural transformation in $\GMap(*_{{\mathbb Z}/n}, \calG)_1 $.  
Then $\alpha$ is defined by a map $* \to h \in \calG_1$ such that the following diagram commutes:
 \[  \xymatrix { f(*)  \ar[r]^{f(\sigma)} \ar[d]_{h} & f(*) \ar[d]^{h}  \\
f'(*)  \ar[r]_{f'(\sigma) } &f'(*) }   \]
Therefore we can define $\Phi(\alpha) =\alpha(*) = h$.
The above commutative diagram says that $hf(\sigma)h^{-1} = f'(\sigma)$, and 
so $h$ defines a morphism in $(\Lambda \calG)_1$,
 $$(f(*), f(\sigma)) \overset{h}{\longrightarrow} (f'(*), f'(\sigma))= (h(f(*)) , hf(\sigma)h^{-1}).$$

To show that $\Phi :\GMap(*_{{\mathbb Z}/n}, \calG) \to (\Lambda \calG)_1$ is a functor we  need to verify that the composition is respected.   
Let $\alpha_1:  f \Rightarrow f'$ and $\alpha_2:  f' \Rightarrow f''$ be natural transformations  and $h_1$ and $h_2$ be the 
corresponding morphisms in $(\Lambda \calG)_1$.  We need to check that
$\Phi(\alpha_2 \circ \alpha_1) = \Phi(\alpha_2) \circ \Phi(\alpha_1)$:
\begin{eqnarray*}
\Phi(\alpha_2) \circ \Phi(\alpha_1) & = & (f(*), f(\sigma)) \xrightarrow{h_2 \circ h_1} \left(h_2(h_1f(*)) , h_2(h_1f(\sigma)h_1^{-1}) h_2^{-1} \right) \\
& = & (f(*), f(\sigma)) \xrightarrow{h_2 \circ h_1} \left((h_2h_1)f(*)) , (h_2h_1)f(\sigma)(h_2 h_1)^{-1} \right) \\
 & = & (f(*), f(\sigma)) \xrightarrow{h_2 \circ h_1} (f''(*), f''(\sigma)) \\
& = & \Phi (\alpha_1 \circ \alpha_2)
\end{eqnarray*}
 \end{proof}
 
\begin{prop}\label{full}  For any $n$, the functor
$\Phi$ is injective on objects,  and full and faithful.  So $\Phi$ is an inclusion of    $\GMap(*_n, \calG)$ as a full subcategory of $\Lambda(\calG)$.
\end{prop}

\begin{proof}

The map $\Phi_0 :\GMap(*_{{\mathbb Z}/n}, \calG)_0 \to \Lambda(\calG)_0$ is  injective on objects, since it sends a functor  $f$ to $(f(*), f(\sigma))$ and a functor from $*_{{\mathbb Z}/n}$ is determined by this information.
To prove that $\Phi$ is full,  we consider arbitrary $f$ and $f'$, and suppose there is a morphism $\Phi(f) \xrightarrow{h} \Phi(f')$.   So $h$ satisfies
\begin{eqnarray*}
 h(\Phi(f))&=& h(f(*), f(\sigma)) \\
               &=& (hf(*), hf(\sigma)h^{-1}) \\
               &=& (f'(*), f'(\sigma))
\end{eqnarray*}
If $hf(*)= f'(*)$ and $hf(\sigma)h^{-1}= f'(\sigma)$ this means that the diagram
 \[  \xymatrix { f(*)  \ar[r]^{f(\sigma)} \ar[d]_{h} & f(*) \ar[d]^{h}  \\
f'(*)  \ar[r]_{f'(\sigma) } &f'(*) }   \]
commutes.  So $\alpha (*) = h$ defines a  natural transformation between $f$ and $f'$ with $\Phi (\alpha) = h$.  Hence $\Phi$ is surjective on arrow sets so it is full.

Now we show that $\Phi$ is faithful. So suppose that  $\Phi(\alpha) = \Phi(\alpha')$ where $\alpha, \alpha':  f \Rightarrow f'$ are natural transformations between functors $f$ and $f'$.  This means that if $\alpha(*) = h_1$ and $ h_2 = \alpha(*) \in (\Lambda \calG)_1$ then $h_1 = h_2$
 and
$$(h_1 f(*), h_1 f(\sigma) h_1^{-1}) = (f'(*), f'(\sigma))=(h_2 f(*), h_2 f(\sigma) h_2^{-1})   $$
So both $\alpha$ and $\alpha'$  correspond to the same commutative diagram, that is to the same natural transformation between $f$ and $f'$.
Hence the functor $\Phi$ is faithful and the proposition follows.
 \end{proof}

\begin{prop}\label{inertia-and-maps}
For each orbit compact orbigroupoid ${\calG}$, there is an $n$ such that $$\wedge(\calG)\simeq \GMap(*_{{\mathbb Z}/n}, \calG)$$
\end{prop}

\begin{proof}
In proposition \ref{full} we have seen that for any $n$,
$$\Phi:\GMap(*_{{\mathbb Z}/n}, \calG) \to \wedge (\calG)$$
is an inclusion of $\GMap(*_{{\mathbb Z}/n}, \calG)$ as a full subcategory of $\wedge(\calG)$.
To find an isomorphism $\GMap(*_{{\mathbb Z}/n}, \calG) \xrightarrow{\cong}\wedge(\calG)$ we need to establish conditions on $n$ that will guarantee that $\Phi$ surjection on objects as well.

For each $x\in\calG_0$ let $V_x$ be a neighbourhood of $x$ as in Remark \ref{D:Etaleproper}.
Then the quotients of these $V_x$ form an open cover of the quotient space $\calG_1/\calG_0$.
Since $\calG$ is orbit compact, there is a finite subset $V_{x_1},\ldots,V_{x_m}$ such that their quotients cover the quotient space.
This means that each point in $\calG_0$ is in the orbit of a point in $\bigcup_{i=1}^mV_{x_i}$.
So all points have isotropy groups that are conjugate to the isotropy groups of points in $\bigcup_{i=1}^nV_{x_i}$.
Now for each $i=1,\ldots,n$ all points in $V_{x_i}$ have an isotropy group which is a subgroup of $G_{x_i}$, so we only need
consider the groups $G_{x_1},\ldots,G_{x_m}$ and all these groups are finite.

A map $f$   in $ \GMap(*_{{\mathbb Z}/n}, \calG)$ is given by a point $x = f(*)$ 
and an arrow $g = f(*) \in G_x$, and we have such a pair  for any $g \in G_x$ such that $g^n = 1$.  Hence  $\Phi:\GMap(*_{{\mathbb Z}/n}, \calG) \to \wedge (\calG)$ will be surjective on objects if every isotropy element $h \in \calG$ has order dividing $n$.
Since each isotropy subgroup is conjugate to a  subgroup of a structure group $G_{x_i}$,   we can define $n$ to be the smallest common multiple of the orders of the structure groups $G_{x_1},\ldots,G_{x_m}$.
\end{proof}

\section{Conclusion and Future Work}
Although the $\GMap$ construction described in this paper is very useful in understanding how the isotropy structure comes about on the mapping space, it is not sufficient as a mapping object for orbispaces.    We have mentioned that $\GMap(\calG,\calH)$ is not necessarily an orbigroupoid (i.e., \'etale and proper), and that even when it is, 
 it is not invariant under Morita equivalence.

In a forthcoming paper, two of the authors continue this project  and show how to define a mapping groupoid  which is Morita invariant,  and use this to show that the   bicategory of orbit compact orbispaces is Cartesian closed.  Producing this groupoid requires some careful considerations of how to create a small category which can be inverted, rather than working with all essential equivalences, and in how to put the topology in place on top of the generalized maps and 2-cells.  The groupoids obtained this way are  homotopy colimits of the $\GMap(\calG,\calH)$ groupoids described in this paper,
where the homotopy colimit is taken over a diagram of groupoids representing the same orbispace as $\calG$.
We will even see that the inclusions of the $\GMap(\calG,\calH)$ groupoids into the larger groupoid are all fully faithful
and all connected components of the object and arrow spaces  of the mapping spaces as orbispaces
are open subsets of the $\GMap(\calG,\calH)$ groupoids.  So all the local structure on the true orbispace  is obtained from the $\GMap(\calG,\calH)$ 
groupoids as illustrated here.    Approaching orbispaces via groupoids allows us to get a more concrete understanding of what is going on with this category and its maps.

 \end{document}